\documentclass{ws-ijtaf}
\usepackage{subcaption}
\usepackage{amsmath}
\usepackage{mathtools}
\usepackage{etoolbox}
\DeclarePairedDelimiter{\abs}{|}{|}
\DeclarePairedDelimiter{\norm}{\lVert}{\rVert}

\DeclareFontFamily{U}{matha}{\hyphenchar\font45}
\DeclareFontShape{U}{matha}{m}{n}{
<-6> matha5 <6-7> matha6 <7-8> matha7
<8-9> matha8 <9-10> matha9
<10-12> matha10 <12-> matha12
}{}
\DeclareSymbolFont{matha}{U}{matha}{m}{n}

\DeclareFontFamily{U}{mathx}{\hyphenchar\font45}
\DeclareFontShape{U}{mathx}{m}{n}{
<-6> mathx5 <6-7> mathx6 <7-8> mathx7
<8-9> mathx8 <9-10> mathx9
<10-12> mathx10 <12-> mathx12
}{}
\DeclareSymbolFont{mathx}{U}{mathx}{m}{n}

\DeclareMathDelimiter{\vvvert} {0}{matha}{"7E}{mathx}{"17}%

\DeclarePairedDelimiterX{\snorm}[1]
{\vvvert}
{\vvvert}
{\ifblank{#1}{\:\cdot\:}{#1}}

\usepackage{bm}% bold math
\usepackage{braket}

\DeclareMathOperator {\diag}{diag}
\DeclareMathOperator {\var}{Var}  % Variance
\DeclareMathOperator {\Exp}{E}    % Expectation
\DeclareMathOperator {\Le}{d_L}     % Levy distance
\DeclareMathOperator {\K}{d_K}      % Kolomogorov distance
\DeclareMathOperator{\Prb}{P}     % Probability
\DeclareMathOperator {\rank}{rank}
\DeclareMathOperator {\Tr}{Tr}
\usepackage{amsmath}	% required for `\aligned' (yatex added)
\begin{document}
\markboth{Y. Akama}{Equi-correlation, the bulk and the largest
eigenvalues of correlation matrices, and scaling}
\title{Correlation matrix of equi-correlated normal population:
fluctuation of the largest eigenvalue, scaling of the bulk eigenvalues,
and stock market}
\author{Yohji Akama}
\address{Department of Mathematics, Graduate School of Science, Tohoku
University, Aramaki, Aoba, Sendai, 980-8578, Japan.\\
\email{yoji.akama.e8@tohoku.ac.jp}}

\maketitle
\begin{history}
\received{(Day Month Year)}
\revised{(Day Month Year)}
\end{history}
\begin{abstract}
Given an $N$-dimensional sample of size $T$ and form a sample
correlation matrix $\mathbf{C}$. Suppose that $N$ and $T$ tend to infinity
with $T/N $ converging to a fixed finite constant $Q>0$. If the
population is a factor model, then the eigenvalue distribution of $\mathbf{C}$
almost surely converges weakly to Mar\v{c}enko-Pastur distribution such
  that the
index is $Q$ and the scale parameter is the limiting ratio of the
 specific variance to the $i$-th variable
$(i\to\infty)$.  For an $N$-dimensional normal population with equi-correlation
  coefficient $\rho$, which is a
one-factor model, for the largest eigenvalue $\lambda$ of $\mathbf{C}$, we prove
that $\lambda/N$ converges to the equi-correlation coefficient $\rho$
almost surely. These results suggest an important role of an
equi-correlated normal population and a factor model in (Laloux et al. Random matrix theory
and financial correlations, \emph{Int. J. Theor. Appl. Finance}, 2000):
the histogram of the eigenvalue of sample correlation matrix of the
returns of stock prices fits the density of Mar\v{c}enko-Pastur
distribution of index $T/N $ and scale parameter
  $1-\lambda/N$. Moreover, we provide the limiting distribution of the largest eigenvalue of a sample
covariance matrix of an equi-correlated normal population.  We discuss
  the phase transition as to the decay rate of the
  equi-correlation coefficient in $N$.
\end{abstract}

\section{Introduction}

For risk management and asset allocation, a sample correlation matrix is
essential. One of the pillars of Markowitz portfolio
theory~\citep{markowitz1959portfolio,elton2014modern} is the study of correlation (or covariance) matrices, e.g.,
the optimal weight of each financial asset, given a set of assets with
average returns and risks, so that the overall portfolio offers the best
return for a fixed level of risk or, alternatively, the smallest risk
for a given overall return. For these optimizations, the eigenvalues of
covariance/correlating matrices are important. However, the eigenvalues
of a sample covariance matrix/correlation matrix in finance are
dispersed with noise, so various ways to eliminate noisy eigenvalues
from them are studied by \cite{10.1142/S0219024900000255},
\cite{10.1214/07-AOS559}, \cite{Ledoit110, 10.1214/12-AOS989},
\cite{ledoit11:_eigen}, \cite{10.1214/17-AOS1601}, to cite a few. For
this purpose, \cite{10.1142/S0219024900000255} proposed to approximate
heuristically the distribution of noisy eigenvalues of a sample
covariance/correlation matrix, with \emph{Mar\v{c}enko-Pastur
distribution}~\citep{mar} of random matrix theory. Their idea influenced
 on the statistical methods on vaccine design~\citep{10.1371/journal.pcbi.1006409}.
We will discuss the
heuristic of \cite{10.1142/S0219024900000255}
mathematically, with an \emph{equi-correlated normal population}, which plays an important role in various risk
models such as \emph{dynamic equicorrelation}~\citep{de}
 and \emph{constant correlation model}~\citep{elton2014modern}.

First of all, we fix the notation.
Let $\bm{X}_1,\ldots,\bm{X}_T  \in{{\mathbb R}}^N$ be a random sample.
By the \emph{data matrix}, we mean ${\mathbf{X}}=[\bm{X}_1,\ldots,\bm{X}_T
 ]\in{{\mathbb R}}^{N\times T}$. Let ${{1\le i\le N}}$ and ${{1\le t\le  T}}$. The $i t$-entry of ${\mathbf{X}}$ is denoted by
 $x_{i t}$, and the $i$-th row by 
 $\bm{x}_i\in{{\mathbb R}}^{1\times  T}$. For the sample average of $\bm{x}_i$, write $\overline{x}_i=T^{-1}{\sum_{t=1}^ T} x_{i t}$.  Let ${\overline{\bm{x}}}_i:=[\overline{x}_i,\ldots,\overline{x}_i]\in
 {{\mathbb R}}^{1\times T}$ and
 $\overline{\mathbf{X}}:=[\overline{x}_1,\ldots,\overline{x}_N]^\top [1,\ldots,1]\in{{\mathbb R}}^{N\times T}$.
The \emph{sample correlation matrix} is, by definition,
\begin{align*}
\mathbf{C}&={\mathbf{Y}}{\mathbf{Y}}^\top\in{{\mathbb R}}^{N\times N}\ \text{where}\  {\mathbf{Y}}^\top=\left[\frac{(\bm{x}_1-{\overline{\bm{x}}}_1)^\top}{\norm{\bm{x}_1-{\overline{\bm{x}}}_1}},\ldots,\frac{(\bm{x}_N -{\overline{\bm{x}}}_N)^\top}{\norm{\bm{x}_N -{\overline{\bm{x}}}_N}}\right].
\end{align*}
Here, the notation $\norm{.}$ is the Euclidean norm. 
The \emph{sample covariance matrix} is denoted by
$\mathbf{S}=T^{-1}{\mathbf{X}}{\mathbf{X}}^\top\in{{\mathbb R}}^{N\times N}$.
By a \emph{distribution function}, we mean a nondecreasing, right-continuous function $F:{{\mathbb R}}\to[0,1]$ such that $\lim_{x\to-\infty} F(x)=0$ and $\lim_{x\to+\infty} F(x)=1$. 

We review the results on eigenvalue distributions from random matrix theory.
The \emph{empirical spectral distribution} of a symmetric matrix
${\mathbf{P}}\in{{\mathbb R}}^{N\times N}$ is, by definition,
 \begin{align*}
  F^{{\mathbf{P}}}(x)=\dfrac{1}{N}\#\Set{i\, \colon\, {{1\le i\le N}},\ \lambda_i({\mathbf{P}})\le x}\quad (x\in{{\mathbb R}})
 \end{align*}
where $\lambda_1({\mathbf{P}}),\ldots,\lambda_N({\mathbf{P}})$ are the eigenvalues of ${\mathbf{P}}$, indexed in nonincreasing order.
According to~\citep[p.~10]{yao.15:_large,bai10:spect}, \emph{Mar\v{c}enko-Pastur
  distribution  with the sample size to dimension ratio index $Q>0$ and scale parameter $\sigma^2\in(0,\,\infty)$} has a density
\begin{align}
\frac{Q}{2\pi \sigma^2 x}\sqrt{(\sigma^2b_Q-x)(x-\sigma^2a_Q)} 1_{x\in[\sigma^2a_Q,\,\sigma^2b_Q]}\label{density}
\end{align}
and has a point mass of value $1 - Q$ at the origin if $Q < 1$.  Here,
$a_Q={(1-\sqrt{1/Q})^2}$, $$b_Q={(1+\sqrt{1/Q})^2},$$ and $1_{\cdots}$ denotes the indicator
function. Mar\v{c}enko-Pastur distribution has expectation
$\sigma^2$. The distribution function of Mar\v{c}enko-Pastur
distribution of the sample to dimension ratio index $Q$ and scale parameter
$\sigma^2$ is denoted by $\mathrm{MP}_{Q,\sigma^2}$. We write $\mathrm{MP}_{Q}$ for
$\mathrm{MP}_{Q,1}$. Clearly, $\mathrm{MP}_{Q,\sigma^2}(x)=\mathrm{MP}_{Q}(x/\sigma^2)$ for all
$x\in {{\mathbb R}}$. We say a sequence $(F_i)_i$ of distribution functions
\emph{converges weakly} to a function $F$, if $F_i(x)\to F(x)$ for every $x\in{{\mathbb R}}$ where
$F$ is continuous.
\begin{proposition}\label{prop:MP law}
Let $\set{x_{i t} \colon\, i,t\ge1}$
be a double array of \emph{i.\@i.\@d.\@} random variables with $\Exp\abs{x_{11}}^2 =\sigma^2\in(0,\,\infty)$.
  Suppose ${{N,T\to\infty,\ {T/N \to Q}}}\in (0,\, \infty)$. Then, almost surely,
 \begin{enumerate}
  \item \label{assert:MP LT}
	$F^{\mathbf{S}}$ converges weakly to $\mathrm{MP}_{Q,\sigma^2}$. \citep[p.~12]{yao.15:_large}
	
  \item \label{assert:MP correlation}
       $F^{\mathbf{C}}$ converges weakly to $\mathrm{MP}_{Q}$.
	\citep[Theorem~1.2]{jiang}
 \end{enumerate}
\end{proposition}

However, the eigenvalues of financial correlation matrices distribute similarly
as the random matrix theory suggests,
but somehow differently.
\cite{10.1142/S0219024900000255} considered a sample
correlation matrix $\mathbf{C}$ ``associated to the time series of the different
stocks of the S\&P500 (or other major markets)'' where $N$ is the
number of assets and $T$ is the number of observations.
\cite{10.1142/S0219024900000255} found that
the histogram of the eigenvalues of the financial, sample correlation
matrix $\mathbf{C}$ fits not the density of $\mathrm{MP}_{T/N }$, but a \emph{scaled}
density, specifically, the density of
\begin{align}\label{shrank MP}\mathrm{MP}_{\frac{T}{N},\,1-\frac{\lambda_1(\mathbf{C})}{N}}.\end{align}
In addition to the fitting result, for the first eigenvalue $\lambda_1(\mathbf{C})$ of $\mathbf{C}$, Laloux et
al.\@ remarked the following:
\begin{quote}The corresponding eigenvector
is, as expected, the ``market'' itself, i.\@e.\@, it has roughly equal
components on all the $N$ stocks.
\end{quote}
\cite{10.1142/S0219024900000255} interpreted the support of the
density function of the fitting Mar\v{c}enko-Pastur distribution, as the
interval of noisy eigenvalues of the sample correlation matrix
$\mathbf{C}$.
Based on the findings of \cite{10.1142/S0219024900000255}, 
in order to improve risk management, \cite{BUN20171}
proposed ``\emph{eigenvalue clipping}: all eigenvalues in bulk of the fitting Mar\v{c}enko-Pastur
spectrum are deemed
as noise and thus replaced by a constant value whereas the principal components outside of the
bulk (the spikes) are left unaltered.''
As portfolios of hundreds
or thousands of assets becomes the target of research,  many approaches have been proposed to deal
with the problem of dimensionality.  By  the limiting
regime ${{N,T\to\infty,\ {T/N \to Q}}}\in (0,\infty)$ of Proposition~\ref{prop:MP law}, we
would like to understand mathematically why the heuristic of
\cite{10.1142/S0219024900000255} is meaningful.

Let us hypothesize about the population of the stock prices Laloux et
al.\@ considered, from their remark on the eigenvector that ``has roughly equal components
on all the $N$ stocks.'' It is comparable to an eigenvector $[1,\ldots, 1]^\top$  corresponding to the
largest eigenvalue
$\lambda_1({\mathbf{\Sigma}}_\rho)$ of an \emph{equi-correlation matrix}
\begin{align*}
{\mathbf{\Sigma}}_\rho=\begin{bmatrix}
1 & \rho & \cdots & \rho\\
\rho & 1 & \cdots & \rho\\
\vdots & \vdots & & \vdots\\
\rho &\rho &\cdots & 1
\end{bmatrix}\in {{\mathbb R}}^{N\times N}\qquad(0\le\rho<1),
\end{align*}
where
\begin{align}\lambda_1({\mathbf{\Sigma}}_\rho)=(N-1)\rho +1
\ge\lambda_2({\mathbf{\Sigma}}_\rho)=\cdots=\lambda_N ({\mathbf{\Sigma}}_\rho)=1-\rho.\label{ev}
\end{align}
\cite{10.2307/4616081} studied hypothesis tests
on $\rho=0$ vs. $\rho>0$.
\cite{ishii21:_hypot} developed nonparametric
 tests of the high-dimensional covariance structures such as a scaled
 identity matrix, a diagonal matrix and ${\mathbf{\Sigma}}_\rho$ for $N>T$ cases.

In particular, when $\rho>0$, $\lambda_1({\mathbf{\Sigma}}_\rho)$ is \emph{unbounded}
with respect to $N$, and thus there is a `spectral gap' between
$\lambda_1({\mathbf{\Sigma}}_\rho)=(N-1)\rho +1$ and $\lambda_2({\mathbf{\Sigma}}_\rho)=1-\rho$.
We presume that the population is a normal population with the population correlation matrix
${\mathbf{\Sigma}}_\rho$.

Then, for an equi-correlated normal population, we established the following scaling theorem for the limiting distribution of
the \emph{bulk eigenvalues} of the sample correlation matrix $\mathbf{C}$:
\begin{proposition}[\protect{\citet{akama22:_dichot_behav_of_guttm_kaiser}}]\label{prop:AH22}Let $\mathbf{C}$ be a sample correlation matrix formed from an $N$-dimensional
 normal population such that the population correlation matrix is
 ${\mathbf{\Sigma}}_\rho$ with deterministic $\rho\in[0,\,1)$. Suppose
 ${{N,T\to\infty,\ {T/N \to Q}}}\in(0,\,\infty)$. Then, almost surely, the empirical spectral
 distribution $F^{\mathbf{C}}$ converges weakly to
 $$\mathrm{MP}_{Q,\, 1-\rho}.$$\end{proposition}

 It is worth noting that 
 the limiting spectral distribution~(LSD)
$\mathrm{MP}_{Q,\, 1-\rho}$ is similar to Mar\v{c}enko-Pastur
 distribution~\eqref{shrank MP} $\mathrm{MP}_{T/N ,\, 1-\lambda_1(\mathbf{C})/N}$  where the histogram of the eigenvalues of a financial
 correlation matrix fits to the density of \eqref{shrank MP}. The index
 $Q$ and the scale parameter $1-\rho$
 correspond to $T/N $ and $1-\lambda_1(\mathbf{C})/N$, respectively.
 
By comparing \eqref{shrank MP} to Proposition~\ref{prop:AH22}, we will show that the largest eigenvalue $\lambda_1(\mathbf{C})$ of the sample
correlation matrix $\mathbf{C}$ divided by the order $N$ of $\mathbf{C}$ converges almost
surely to the equi-correlation coefficient $\rho$, as follows. 
Hereafter, ${\stackrel{a.s.}{\to}}$ and ${\stackrel{D}{\to}}$
mean the almost sure convergence and the convergence in distribution
respectively. Unless otherwise stated, the limiting regime is ${{N,T\to\infty,\ {T/N \to Q}}}\in(0,\,\infty)$.
\begin{theorem}\label{thm:main}Suppose the assumptions of
 Proposition~\ref{prop:AH22}. Then,
 \begin{enumerate}
  \item  \label{consistency}
 $\lambda_1(\mathbf{C})\big/{N}\mathrel{\stackrel{a.s.}{\to}}\rho$.
\item\label{fluctuation}
 If, moreover, the population is $\mathrm{N}_N(\bm{0},\ {\mathbf{\Sigma}}_\rho)$ and
 $\rho>0$, then
 \begin{align*}
   \frac{\lambda_1(\mathbf{S}) - \tau}{\varsigma}\ {\stackrel{D}{\to}}\ \mathrm{N}(0,1)
 \end{align*}
 where\begin{align*}\tau=\frac{\left((N -1) \rho +1\right)
   \left(\left(1+(T-1 )N \right) \rho +N -1\right)}{N T \rho},\quad \varsigma=\frac{(N -1) \rho +1}{\sqrt{2 T}}.       
     \end{align*}
 \end{enumerate}
\end{theorem}

From the first assertion, i.e., the consistency of $\lambda_1(\mathbf{C})/N$, of Theorem~\ref{thm:main}, Proposition~\ref{prop:AH22} suggests that an
equi-correlated normal population plays a role in Laloux et al.'s fitting
result to the density of \eqref{shrank MP}:
\begin{corollary}\label{cor:soundness} Assume the assumptions of
 Proposition~\ref{prop:AH22}. Then,
\begin{align*}
F^{\mathbf{C}}(x)-\mathrm{MP}_{\frac{T}{N},\, 1-\frac{\lambda_1(\mathbf{C})}{N}}(x)\mathrel{\stackrel{a.s.}{\to}}0\quad(x\ne0).
\end{align*}
\end{corollary}
Once the largest eigenvalue $\lambda_1(\mathbf{C})$ of $\mathbf{C}$ is found,
Corollary~\ref{cor:soundness} can suggest the support of
the density of the fitting Mar\v{c}enko-Pastur distribution to clean
up~(``eigenvalue clipping''), and then we
can eliminate the eigenvalues in the support to let a small, but
important eigenvalue~(e.\@g., Markowitz solution) show up. Moreover, $\lambda_1(\mathbf{C})$ has been concerned by
\cite{doi:10.1177/001316448104100102}.
              
The rest of the paper is organized as follows. In the next section, to
  generalize Proposition~\ref{prop:AH22}, we consider suitable factor
  models. Then, we establish the limiting spectral distribution
  of the sample correlation matrix $\mathbf{C}$ generated from such a
  factor model.
In Section~\ref{sec:real data analysis}, for returns of S\&P stock
 prices (2012-1-4/2021-12-31),
 we compute the largest eigenvalues
 of financial correlation matrices $\mathbf{C}$ divided by the order, and
 estimate the time average $\overline\rho$ of equi-correlation coefficient by
 GJR GARCH~\citep{10.2307/2329067} with the correlation structure being
 `equicorrelation'~\citep{de}. We show that
 $\lambda_1(\mathbf{C})/N$ predicts very well $\overline\rho$.
 In Section~\ref{sec:BBP},
 for a nonnormal population with equicorrelation coefficient $\rho=\rho_N$ decaying in
 $N$, we review the consistency and the asymptotic normality of
 $\lambda_1(\mathbf{S})$, by citing
 \cite{yata09:_pca_consis_non_gauss_data}. Then, for the limiting
 distribution of an appropriate scaling of a translation of $\lambda_1(\mathbf{S})$ of equi-correlated normal population,
 we discuss the phase transition from the normal population to
 \emph{Tracy-Widom distribution}~\citep{baik05:_phase}, as to the decay rate of $\rho_N$.  
The proofs of Theorem~\ref{thm:main},
 Corollary~\ref{cor:soundness}, and Theorem~\ref{thm:equilsd2} are postponed in
Section~\ref{proof:main}, 
 Section~\ref{proof:soundness}, and  Section~\ref{proof:equilsd2},  respectively, for the sake of the presentation.

\section{From equi-correlated normal population to factor models in limit}\label{sec:factor model in limit}

 Proposition~\ref{prop:AH22} was proved by the following well-known fact~\cite[(3.21)]{fan19:_larges}:
Assume that $\bm{X}_1,\ldots,\bm{X}_T
\stackrel{\mbox{\textrm{i.\@i.\@d.\@}}}{\sim} \mathrm{N}_N
 (\bm{0},\,{\mathbf{\Sigma}}_\rho)$, $0\le \rho<1$, and
 $[\bm{X}_1,\ldots,\bm{X}_T]=[x_{i t}]_{it}\in{{\mathbb R}}^{N\times T}$.  Then, there are independent
 standard normal random variables ${{f}}_t$,  $e_{i t}$ $({{1\le i\le N}},\
 {{1\le t\le T}})$ such that
 \begin{align}\label{decomp} 
  x_{i t}=  \sqrt{\rho}{{f}}_t + \sqrt{1-\rho}e_{i t} .
 \end{align}
Actually, an equi-correlated normal population is a \emph{factor
  model}~\citep{https://doi.org/10.1111/1468-0262.00392,mulaik}
  consisting of 1 factor. See Example~\ref{eg:ENP}.
  Beside this 1-factor model, a 3-factor model and a 5-factor model are
  employed by
  \cite{FAMA19933,FAMA20151} to describe stock returns  in asset pricing and portfolio management.

\cite{https://doi.org/10.1111/1468-0262.00392}
studied an inferential theory for a general factor model of large
dimensions to which classical factor analysis does not apply, by using
limiting regime $N,T\to\infty, \sqrt{N}/T\to0$. He derived 
the limiting distributions of the estimated factors and estimated factor
loadings,  the asymptotic normality of
estimated common components, and their convergence rates.
 The convergence rate of the
estimated factors and factor loadings can be faster than that of the
estimated common components. These achievements are obtained under liberal
conditions that allow for \emph{correlations} and heteroskedasticities. Stronger results are obtained when the idiosyncratic errors
are serially uncorrelated and homoskedastic. A necessary and sufficient
condition for consistency of the estimated factors is derived for large $N$ but fixed $T$.

We will study a factor model in the limiting regime $N,T\to\infty,\ T/N\to Q$.
First, we recall a factor model~\citep{https://doi.org/10.1111/1468-0262.00392,mulaik}:

\begin{definition}\label{def:factor model}For a nonnegative integer $K$, by a \emph{$K$-factor model}, we mean a data matrix
 \[
 \left[x_{i t}\right]_{i t}\in{{\mathbb{R}}}^{N\times T}
 \]
 such that
\begin{align*}
&x_{i t} = \mu_i + \bm{\ell}_i {{\bm{f}}}_t + e_{i t}\quad(1\le i\le N,\ 1\le
	  t\le  T),
	  \\
	  &\mbox{$\mu_i$ is a deterministic constant,}
	  \qquad\qquad\qquad\qquad\qquad\  \qquad\qquad\qquad\qquad(\mbox{\emph{mean}})	  \nonumber
	  \\
&\mbox{${\bm{\ell}}_i=[\ell_{i1},\ldots,\ell_{i K}]$ is a $K$-dimensional row
	  vector,} \qquad\qquad\qquad(\mbox{\emph{factor loadings}})
	  \\
&\mbox{${{\bm{f}}}_t=[f_{t 1},\ldots,f_{t k}]^\top$ is a $K$-dimensional
	  column vector,} \qquad\quad(\mbox{\emph{common factors}})
	  \\
	  &\mbox{$f_{t k}$ $(1\le t\le  T,\ 1\le k\le K)$ are
	  centered, i.i.d. with $\var f_{t 1}=1$},
	  \\
&\mbox{$e_{i t}$ $(1\le t\le  T)$ are centered,
	  i.i.d., for each $i$ $(1\le i\le N)$},
	  \\
&\mbox{$f_{t k}$ and $e_{i t}$ $(1\le i\le p,\
	  1\le t\le T,\ 1\le k\le K)$ are
	  independent},
	  \ \mbox{and}\\
	  &
	  \psi_i:=\var e_{i t}<\infty\quad\quad\quad\quad\quad\quad\quad\quad\quad\qquad  \qquad\qquad\qquad\qquad(\mbox{\emph{specific
	  variance}}).
 \end{align*}$e_{i t}$'s are called \emph{idiosyncratic components}. $K$
 is the number of factors.
\end{definition}

\begin{example}\label{eg:ENP}
 \label{equicor} The decomposition~\eqref{decomp} is a factor model such
 that $K=1$,   $\mu_i:=0$,
$\ell_{i1}=\sqrt{\rho}$,  
  $\psi_i:=1-\rho$, and both of $f_{t 1}$ and $e_{i t}$ are independent and normal.
\end{example}

 \begin{assumption}\label{ass:conv} $e_{it}$ $(i,t\ge1)$ are \emph{i.i.d.}, and
${\bm{\ell}}_i\to\bm{\ell}$ for some $K$-dimensional row vector $\bm{\ell}$. In this case, define
    \begin{align*}
   \sigma_\infty^2:=\lim_{i\to\infty}\var x_{i t}= \norm{\bm{\ell}}^2  +\psi_1<\infty,
  \qquad
 \rho:=\frac{\norm{\bm{\ell}}^2}{\norm{\bm{\ell}}^2  +\psi_1}.
    \end{align*}
\end{assumption}
Then, $1-\rho$ is the portion of the specific variance to the limiting total
variance $\sigma_\infty^2$ of a variable $x_{it}$ $(i\to\infty)$.

The following generalizes Proposition~\ref{prop:AH22}.
The LSD of a sample correlation matrix $\mathbf{C}$ generated from a factor model is  as follows:

\begin{theorem}\label{thm:equilsd2}  Given a factor model.
Suppose ${{N,T\to\infty,\ {T/N \to Q}}}\in(0,\,\infty)$ and Assumption~\ref{ass:conv}. 
 Then,  almost surely,
	$F^{\mathbf{C}}$ converges weakly to $\mathrm{MP}_{Q,\,1-\rho}$.
\end{theorem}
By Example~\ref{eg:ENP},
this generalizes Proposition~\ref{prop:AH22}.
Theorem~\ref{thm:equilsd2} is proved in Section~\ref{proof:equilsd2}.

\section{The largest eigenvalue of financial correlation matrix}\label{sec:real data analysis}
For a dataset of returns of $N$ S\&P500 stocks~(or other major markets)
for $T$ trading days, \cite{10.1142/S0219024900000255} fitted the histogram of the
eigenvalues of the correlation matrix $\mathbf{C}$, to the density function of
a scaled Mar\v{c}enko-Pastur distribution.

Naturally, the return of a stock is more correlated with the return of
a stock of the same industry classification sector, than with the return
of a stock of a different industry classification sector. We consider
the 11 \emph{global industry classification standard} (GICS) sectors of
S\&P500 stocks. For each GICS, for the dataset of the stock returns for the period
2012-01-04/2021-12-31,  the time series of the equi-correlation
coefficient is computed by GJR GARCH~\citep{10.2307/2329067} with the correlation
structure being \emph{dynamic equicorrelation}~\citep{de} (We abbreviate
this model to `GJR GARCH-DECO'). The program is due to \cite{candila21}.

However, the computation time of the time series of the equi-correlation
coefficient by GJR GARCH-DECO is large, although we
can quickly compute the estimator $\lambda_1(\mathbf{C})/N$ of the
equi-correlation coefficient by a power method. For each
dataset, we compare the time series to $\lambda_1(\mathbf{C})/N$ of the
correlation matrix $\mathbf{C}$.

\def\scl{.22}
\def\scl{.2}
\def\wdt{.32}
\begin{figure}[ht]\centering
   \begin{subfigure}{\wdt\textwidth}\centering
\includegraphics[scale=\scl]{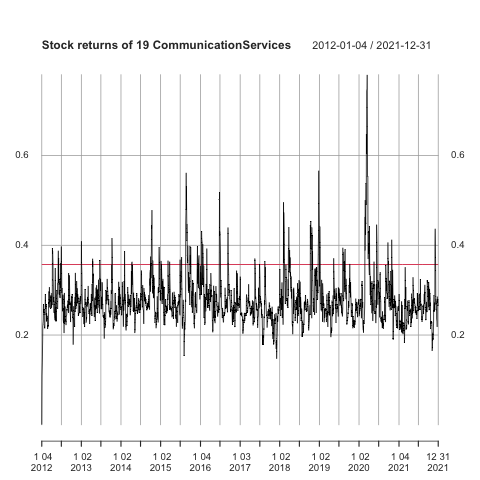}\end{subfigure}
   \hfill\begin{subfigure}{\wdt\textwidth}\centering
\includegraphics[scale=\scl]{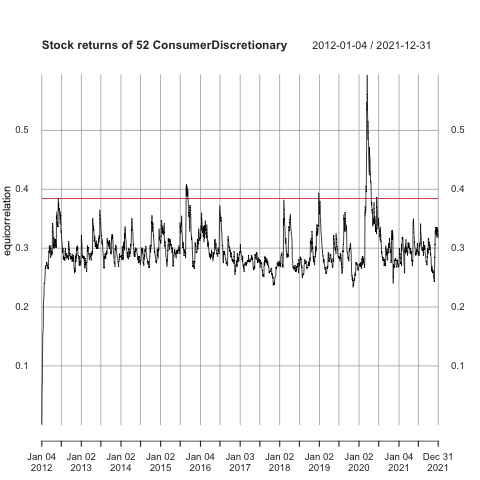}\end{subfigure}
   \hfill\begin{subfigure}{\wdt\textwidth}\centering
\includegraphics[scale=\scl]{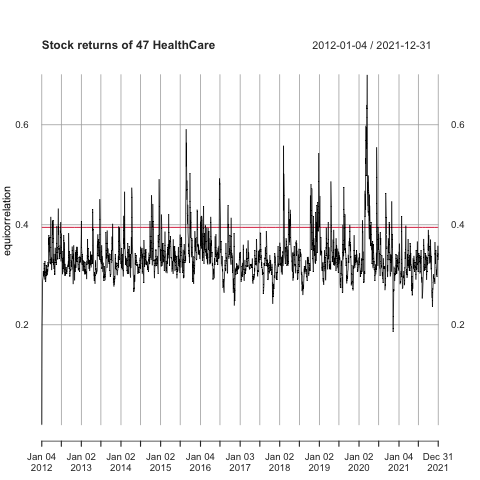}\end{subfigure}
  \hfill\begin{subfigure}{\wdt\textwidth}\centering
\includegraphics[scale=\scl]{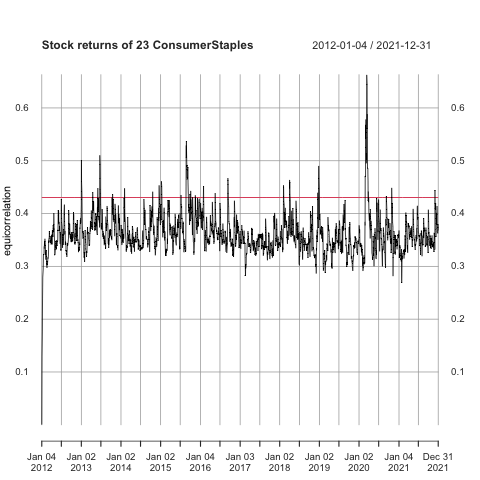}\end{subfigure}
  \hfill\begin{subfigure}{\wdt\textwidth}\centering  
\includegraphics[scale=\scl]{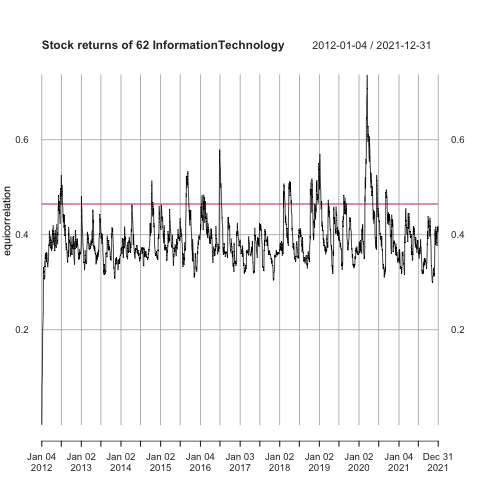}\end{subfigure}
  \hfill\begin{subfigure}{\wdt\textwidth}\centering
\includegraphics[scale=\scl]{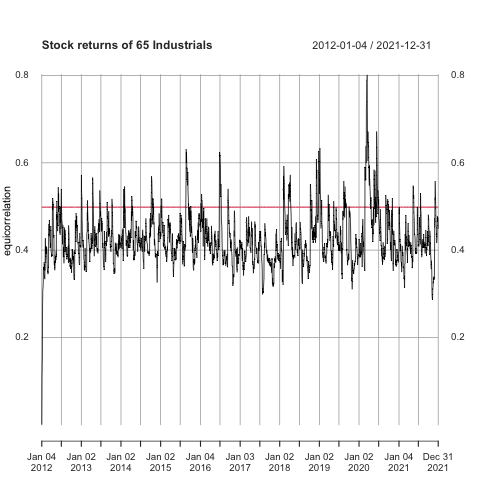}\end{subfigure}
  \hfill\begin{subfigure}{\wdt\textwidth}\centering  
\includegraphics[scale=\scl]{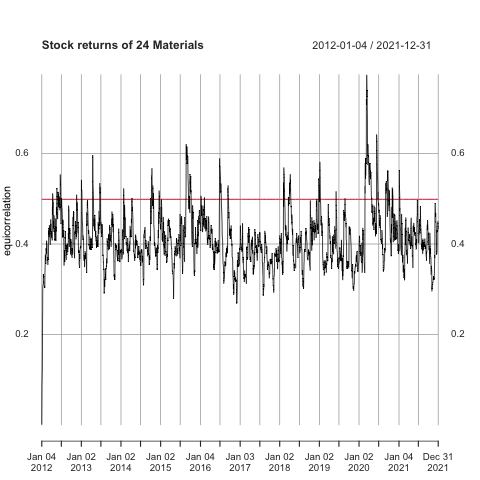}\end{subfigure}
  \hfill\begin{subfigure}{\wdt\textwidth}\centering  
\includegraphics[scale=\scl]{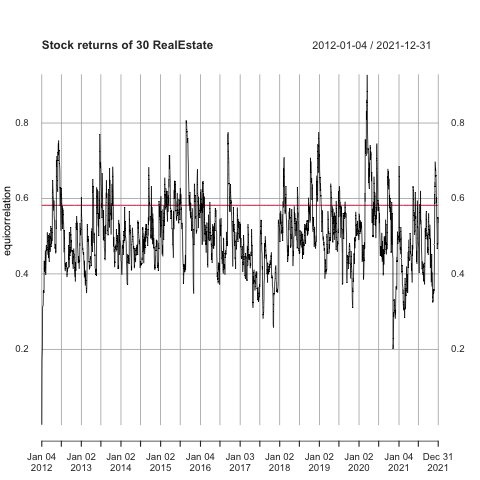}\end{subfigure}
  \hfill\begin{subfigure}{\wdt\textwidth}\centering
\includegraphics[scale=\scl]{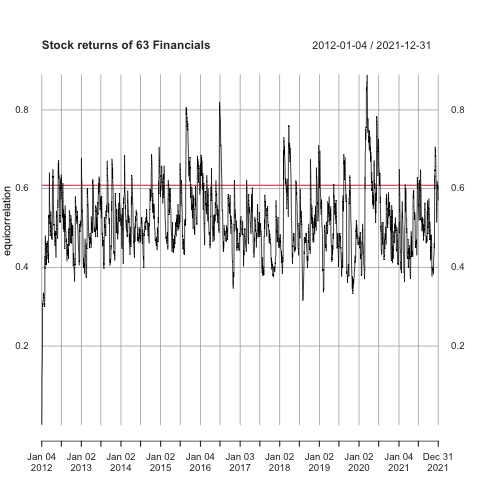}\end{subfigure}
  \hfill\begin{subfigure}{\wdt\textwidth}\centering
\includegraphics[scale=\scl]{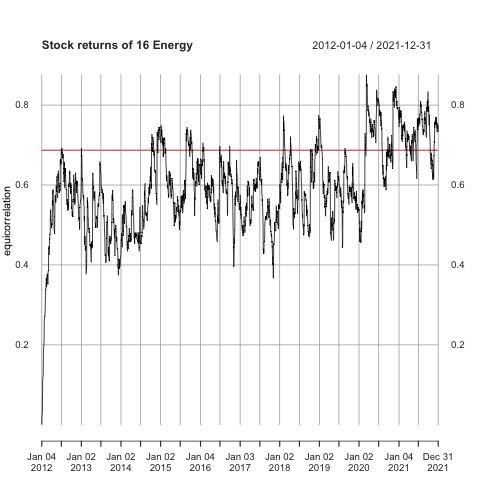}  \end{subfigure}
\hfill\begin{subfigure}{\wdt\textwidth}\centering  
\includegraphics[scale=\scl]{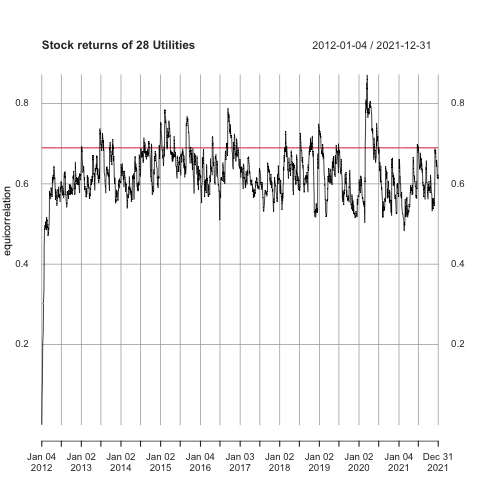}\end{subfigure}
\hfill\begin{subfigure}{\wdt\textwidth}\centering  
\includegraphics[scale=\scl]{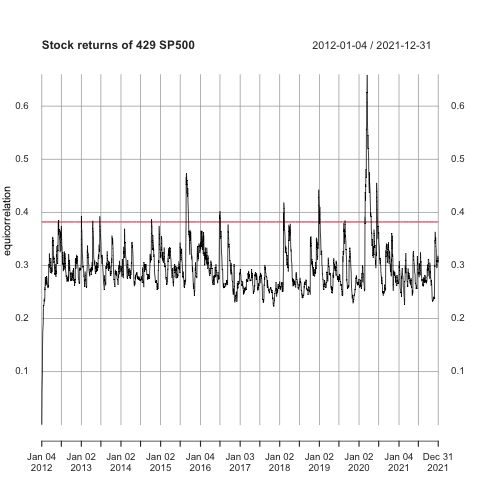}\end{subfigure}
  \caption{Time series of $\rho$ and $\lambda_1(\mathbf{C})/N$ (red line) of the stock returns
  of the 12 datasets of Table~\ref{tbl:gics}.\label{fig:12}}
\end{figure}
Figure~\ref{fig:12} is the time series of equi-correlations for the 12
 datasets. All the time series rise around March 2020, by the world-wide
 sudden crash after growing uncertainty due to the COVID-19 (Coronavirus
 disease 2019) pandemic.   The time series of the equi-correlation
 coefficient of Energy sector kept high, even after the  time 
 series of the equi-correlation of the other 10 sectors of GICS as well as the total
 decrease. According to Figure~\ref{fig:12}, for each dataset,
 the time series is mostly smaller than $\lambda_1(\mathbf{C})/N$. Let us
 consider the time average $\overline{\rho}$ of equi-correlation coefficients
 estimated by GJR GARCH-DECO.  Table~\ref{tbl:gics} is the list of GICS,
 $N/T $, $\lambda_1(\mathbf{C})/N$, $N$, and $\overline{\rho}$, for the period
 2012-01-04/2021-12-31~($T=2516$=trading days $- 1$).  The rows of
 Table~\ref{tbl:gics} are ordered in the increasing order of
 $\lambda_1(\mathbf{C})/N$.
 \def\commserv{.357}
\def\consumerdiscretionary{.384}
\def\healthcare{.394}
\def\consumerstaples{.430}
\def\informationtechnology{.464}
\def\industrials{.498}
\def\materials{.498}
\def\realestate{.581}
\def\financials{.608}
\def\energy{.687}
\def\utilities{.689}
\def\spfivehundred{.381}

\def\avcommserv{.276}
\def\avconsumerdiscretionary{.297}
\def\avconsumerstaples{.361}
\def\avenergy{.595}
\def\avfinancials{.516}
\def\avhealthcare{.334}
\def\avindustrials{.424}
\def\avinformationtechnology{.389}
\def\avmaterials{.410}
\def\avrealestate{.502}
\def\avutilities{.621}
\def\avspfivehundred{.293}

%%%% rounded to the 4 digits
\def\Ccommserv{.0076 }\def\Cconsumerdiscretionary{.0207
}\def\Chealthcare{.0187 }\def\Cconsumerstaples{.0091
}\def\Cinformationtechnology{.0246 }\def\Cindustrials{.0258
}\def\Cmaterials{.0095 }\def\Crealestate{.0119
}\def\Cfinancials{.025 }\def\Cenergy{.0064 }\def\Cutilities{.0111 }\def\Cspfivehundred{.1705 }
\def\commserv{ .3571 }\def\consumerdiscretionary{ .3843
}\def\healthcare{ .3948 }\def\consumerstaples{ .4302
}\def\informationtechnology{ .4648 }\def\industrials{ .4985
}\def\materials{ .499 }\def\realestate{ .5819 }\def\financials{ .6086
}\def\energy{ .6872 }\def\utilities{ .6897 }\def\spfivehundred{ .3819 } 
\def\avcommserv{ .2764 }\def\avconsumerdiscretionary{ .2967 }\def\avhealthare{ .334 }\def\avconsumerstaples{ .3606 }\def\avinformationtechnology{ .3894 }\def\avindustrials{ .4241 }\def\avmaterials{ .4104 }\def\avrealestate{ .5016 }\def\avfinancials{ .5159 }\def\avenergy{ .5954 }\def\avutilities{ .621 }\def\avspfivehundred{ .2939 }
\begin{table}[ht]\centering
\begin{tabular}{ |c|c|l|c|c|c|}
\hline
No&GICS & $N/T $ & $\lambda_1(\mathbf{C})/N$  &$N$ &$\overline{\rho}$\\
\hline
1& Communication Services &\Ccommserv & \commserv &19 & \avcommserv\\
2&Consumer Discretionary &\Cconsumerdiscretionary  & \consumerdiscretionary &52  & \avconsumerdiscretionary\\  
3&Health Care 	&\Chealthcare & \healthcare	&47 & \avhealthcare\\
4&Consumer Staples   &\Cconsumerstaples	& \consumerstaples  &23 & \avconsumerstaples\\
5&Information Technology&\Cinformationtechnology & \informationtechnology & 62 & \avinformationtechnology\\
6&Industrials 	&\Cindustrials	& \industrials &65 & \avindustrials\\
7&Materials &\Cmaterials & \materials & 24 & \avmaterials\\
8&Real estate &\Crealestate & \realestate & 30 & \avrealestate\\
9&Financials 	&\Cfinancials	& \financials&	63 & \avfinancials\\
10&Energy 	&\Cenergy & \energy	&16 & \avenergy\\
11&Utilities &\Cutilities 	& \utilities &  28 & \avutilities\\
 \hline
 12&The total &\Cspfivehundred      & \spfivehundred  &429  &\avspfivehundred\\
 \hline
\end{tabular}
\caption{The stock return datasets of 11 GICS's and the total
 (2012-01-04/2021-12-31), the largest eigenvalue of the correlation
 matrices and the time averages of equicorrelation
 coefficient. $N/T$, $\lambda_1(\mathbf{C})/N$ and $\overline{\rho}$ are rounded.}
\label{tbl:gics}
\end{table}
The last row of Table~\ref{tbl:gics} is for the union dataset of the 11
GICS datasets.  The union dataset is more `heterogeneous' than the 11
GICS datasets. It has $\lambda_1(\mathbf{C})/N$ and $\overline{\rho}$ smaller than the
corresponding $\lambda_1(\mathbf{C})/N$ and $\overline{\rho}$ for most of GICS.  It
means that correlations coefficients between different GICS's are very
small. Figure~\ref{heatmap:gics} are the heat maps with hierarchical
clustering algorithms for the correlation matrices $\mathbf{C}$ of the 12
datasets. According to the last heat map of Figure~\ref{heatmap:gics},
the correlation matrix of the returns of S\&P500 stocks has diagonal
block structures as well as an almost dominant background.

\begin{figure}[ht]\centering
\begin{subfigure}{\wdt\textwidth}\centering
\includegraphics[scale=\scl]{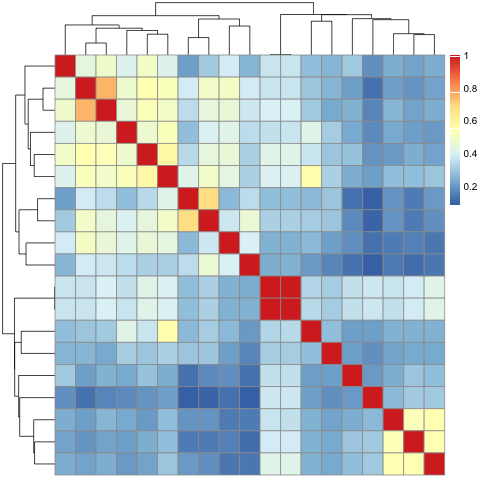}\caption{
 Communication
Services}\end{subfigure}\hfill
\begin{subfigure}{\wdt\textwidth}\centering
 \includegraphics[scale=\scl]{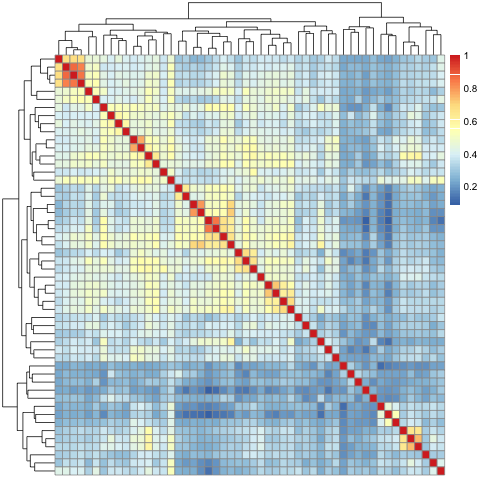}
 \caption{ Consumer Discretionary}
 \end{subfigure}\hfill
 \begin{subfigure}{\wdt\textwidth}\centering
\includegraphics[scale=\scl]{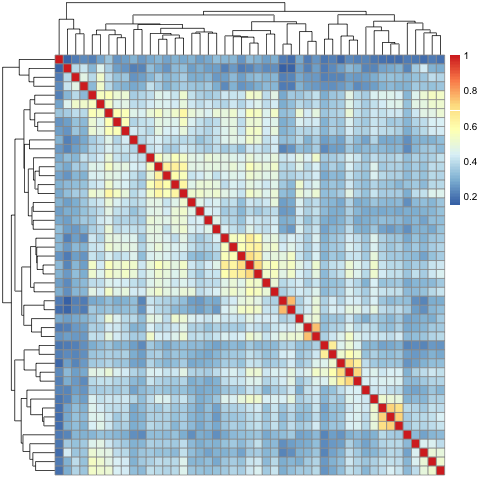}
  \caption{ Health Care}
 \end{subfigure}\hfill
\begin{subfigure}{\wdt\textwidth}\centering
 \includegraphics[scale=\scl]{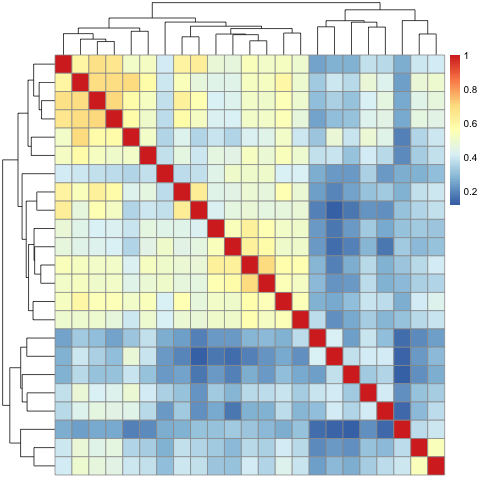}
\caption{ Consumer Staples}\end{subfigure}\hfill
\begin{subfigure}{\wdt\textwidth}\centering
\includegraphics[scale=\scl]{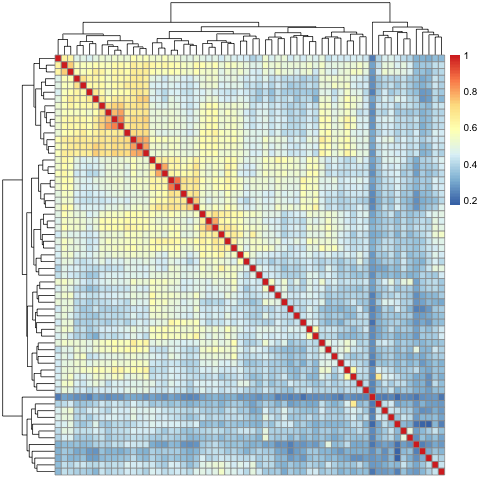}
\caption{ Information Technology}\end{subfigure}\hfill
\begin{subfigure}{\wdt\textwidth}\centering
\includegraphics[scale=\scl]{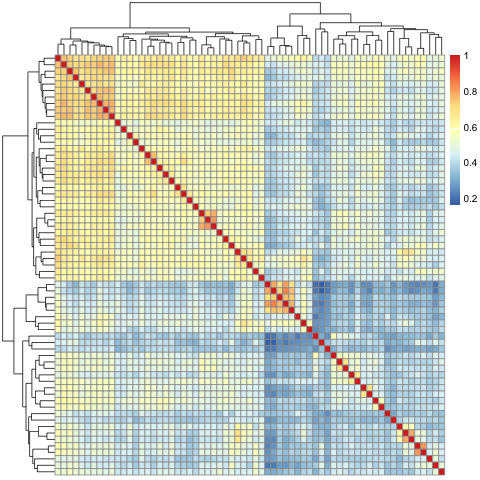}\caption{
 Industrials}
\end{subfigure}\hfill
\begin{subfigure}{\wdt\textwidth}\centering
 \includegraphics[scale=\scl]{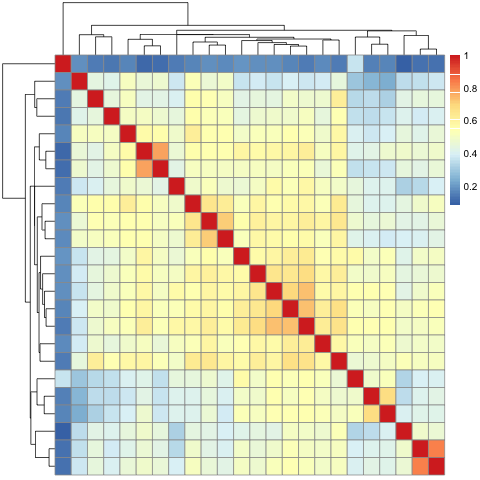}
\caption{ Materials}\end{subfigure}\hfill
\begin{subfigure}{\wdt\textwidth}\centering
 \includegraphics[scale=\scl]{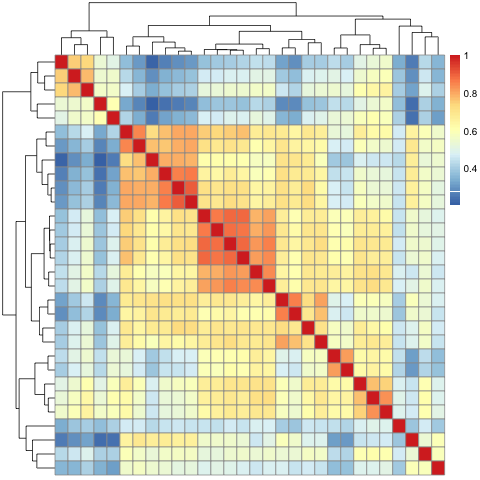}
\caption{ Real Estate}\end{subfigure}\hfill
 \begin{subfigure}{\wdt\textwidth}\centering
  \includegraphics[scale=\scl]{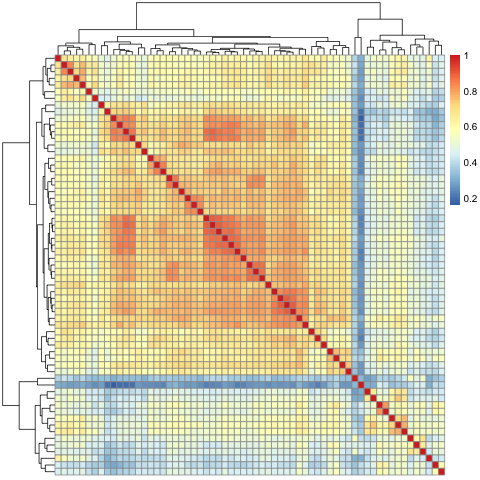}
  \caption{ Financials}
\end{subfigure}\hfill
\begin{subfigure}{\wdt\textwidth}\centering
\includegraphics[scale=\scl]{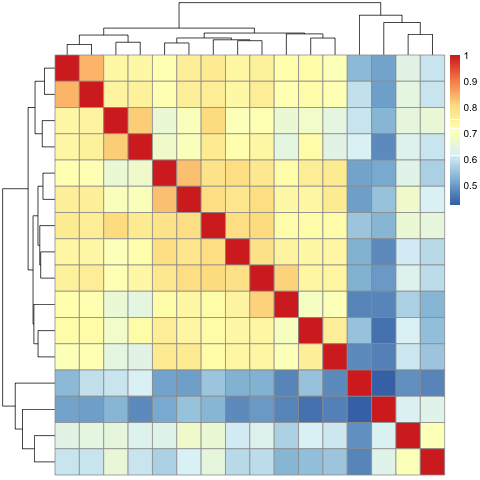}  \caption{ Energy}\end{subfigure}\hfill
\begin{subfigure}{\wdt\textwidth}\centering
\includegraphics[scale=\scl]{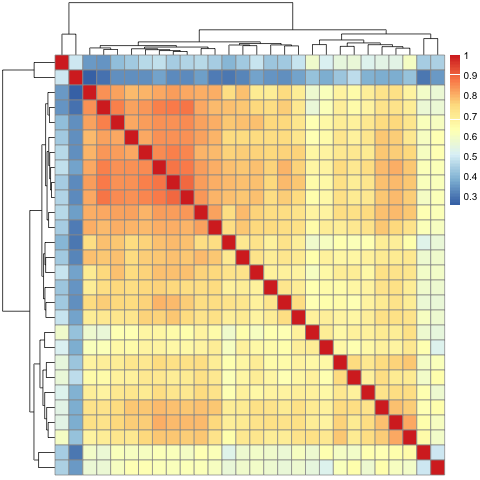}\caption{
Utilities}\end{subfigure}\hfill
\begin{subfigure}{\wdt\textwidth}\centering
 \includegraphics[scale=\scl]{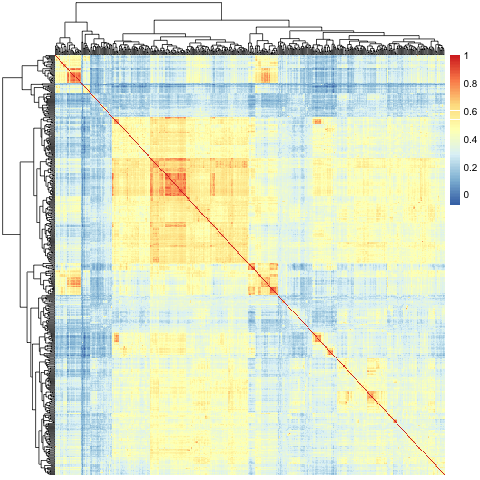}
 \caption{The total}
\end{subfigure}
    \caption{The correlation matrices of the stock returns of
 the 12 datasets of Table~\ref{tbl:gics} and the total.}
    \label{heatmap:gics}
\end{figure}

Although Figure~\ref{heatmap:gics} shows the various structures of $\mathbf{C}$
for the 12 datasets, $\lambda_1(\mathbf{C})/N$ predicts well the
time average $\overline\rho$ of the time series of equi-correlation coefficients computed
by GJR GARCH-DECO with much time. A regression
analysis~(Figure~\ref{fig:pred}) for Table~\ref{tbl:gics} shows:
\begin{align*}
 \overline{\rho}\sim 0.98488\frac{\lambda_1(\mathbf{C})}{N} -
0.07190\quad(\mbox{adjusted $R^2$ being 0.9916}).
\end{align*} 
\begin{figure}[ht]
    \centering
      \includegraphics[scale=.4]{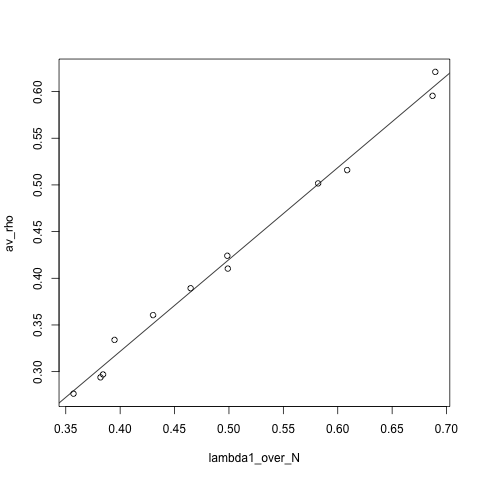}
    \caption{The regression analysis of $\overline\rho$ by
   $\lambda_1(\mathbf{C})/N$.\label{fig:pred}}
\end{figure}

\section{Phase transition of limiting distribution of  largest eigenvalue}\label{sec:BBP}

In our setting, all the correlation coefficients among $N$ variables
being the positive constant $\rho$. It is, however, somehow unnatural in
the limit $ N\to\infty.$ We assume that
\[
 N>T
\]
and  the equi-correlation coefficient
$\rho$  converges to 0 from above as $ N\to\infty$.

For this setting with $N>T$, the consistency $\lambda_1(\mathbf{S})/ N\to\rho$
of the largest eigenvalue $\lambda_1(\mathbf{S})$ and the asymptotic normality
of $\lambda_1(\mathbf{S})$ for a nonnormal population follows from
\cite{yata09:_pca_consis_non_gauss_data}. Moreover, the limiting
distribution of $\lambda_1(\mathbf{S})$ for an equi-correlated normal population
seems to have a phase transition depending on the equi-correlation
coefficient $\rho$, as the limiting distribution of the largest
eigenvalue of $\mathbf{S}$ formed by a spiked population model does so with
respect to the largest eigenvalue of the population covariance
matrix~\citep{baik05:_phase}.

\subsection{Consistency and asymptotic normality of $\lambda_1(\mathbf{S})$ of nonnormal population}
 In
  \citep{yata09:_pca_consis_non_gauss_data}, the consistency and the
  asymptotic normality of 
  $\lambda_1(\mathbf{S})$ are
  proved without assuming the normality of the population.

   \begin{proposition}\label{prop:yataaoshima2009}Let  a  data matrix ${\mathbf{X}}\in{{\mathbb R}}^{N\times T}$ be $[\bm{X}_1,\ldots,\bm{X}_T
	 ]$ where $\bm{X}_t$ $(1\le t\le  T)$ are centered, i.i.d.\@ and $N>T$.
    Suppose that the population covariance matrix ${\mathbf{\Sigma}}$ is $\mathbf{H}\diag(\lambda_1,\ldots,\lambda_N)\mathbf{H}^\top$
    for some orthogonal matrix $\mathbf{H}$, 
    \begin{align}
     &\lambda_1\ge\lambda_2\ge\cdots\ge\lambda_N\ge0,\nonumber\\
     &  \lambda_i = c_i  N^{\delta_i}\ (i=1,\ldots,m),\ \mbox{and}   \label{evya}\\
     &  \mbox{$\lambda_j$ are constants $(j=m+1,\ldots,N)$}.\nonumber
    \end{align}
   Here $m$,  $c_i(>0)$ and $\delta_i$\
    $(\delta_1\ge\cdots\ge\delta_m>0)$ are constants.
    
    Let $\mathbf{Z}:=[\bm{Z}_1,\ldots,\bm{Z}_N
	 ]:=\diag(\lambda_1,\ldots,\lambda_N)^{-1/2}\mathbf{H}^\top {\mathbf{X}}$.   Assume
	 that $\bm{Z}_1,\ldots,\bm{Z}_N$ are independent, the fourth
	 moments of the entries of $\mathbf{Z}$ are uniformly bounded, and no
	  $\bm{Z}_i$ is $\bm{0}$. Moreover suppose
   \begin{enumerate}
    \item[(i)] $N,T\to\infty$ for $i$ such that $\delta_i>1$; and
    \item[(ii)] $N,T\to\infty$ and $ N^{2-2\delta_i}/T\to0$ for $i$ such
	  that $\delta_i\in(0,\,1]$.
   \end{enumerate}
   Then, we have
   \begin{enumerate}
    \item $\lambda_1(\mathbf{S})/\lambda_1=1+o_p(1)$. \cite[Theorem~1]{yata09:_pca_consis_non_gauss_data}
    \item  \label{yata:thm2} Assume also that the largest $m$ eigenvalues
	  of ${\mathbf{\Sigma}}$ are distinct and $\var(z_{ij}^2)=M_i\ (<\infty)$.
	  Then, $\sqrt{T/M_i}(\lambda_1(\mathbf{S})/\lambda_1-1)\mathrel{\stackrel{D}{\to}}
	   \mathrm{N}(0,\,1)$. \cite[Theorem~2]{yata09:_pca_consis_non_gauss_data}
   \end{enumerate}
   \end{proposition}
The condition (ii) is relieved for
\cite[Theorem~2]{yata09:_pca_consis_non_gauss_data} if an independence
condition is further assumed.
   \begin{proposition}[\protect{\cite[Corollary~1]{yata09:_pca_consis_non_gauss_data}}]
 In the assertion~\eqref{yata:thm2} of
    Proposition~\ref{prop:yataaoshima2009}, assume further that
    $z_{ij}$, $1\le i\le N (1\le t\le  T)$ are independent. Then, it
    holds that $\lambda_1(\mathbf{S})/\lambda_1=1+o_p(1)$ if the condition (i)
    and
    \begin{quote}
    (iii) \ $N,T\to\infty$ and $ N^{1-2\delta_i}/T\to0$ for $i$
    such that $\delta_i\in(0,\,1]$.   \end{quote}
   \end{proposition}

By \eqref{ev}, ${\mathbf{\Sigma}}_\rho$ does not meet the condition~\eqref{evya}. According to Kazuyoshi
   Yata, the conclusions of \cite[Theorem~1]{yata09:_pca_consis_non_gauss_data} and
   \cite[Theorem~2]{yata09:_pca_consis_non_gauss_data} indeed hold for
   ${\mathbf{\Sigma}}={\mathbf{\Sigma}}_\rho$. Therefore, we get the following:

  \begin{theorem}\label{thm:yataaoshima}
Suppose that a population is possibly nonnormal, the population
   covariance matrix is ${\mathbf{\Sigma}}_\rho$ $(\rho>0)$, and
  $N,T\to\infty, {T/N \to Q}\in (0,\,\infty)$. Assume 
   \begin{enumerate}
    \item
 the setting of
	  \cite[Theorem~1]{yata09:_pca_consis_non_gauss_data}  and
 $\lim_{ N\to\infty} N^{1/2}\rho=\infty$; or
	 
    \item \label{nrho} the setting of
	  \cite[Corollary~1]{yata09:_pca_consis_non_gauss_data} and
	   $\lim_{ N\to\infty}N\rho=\infty$.
    \end{enumerate}
Then, we have
	 ${\lambda_1(\mathbf{S})}/{N}\mathrel{\stackrel{P}{\to}}\rho$ and the asymptotic normality of
	  $\lambda_1(\mathbf{S})$.
  \end{theorem}
  
\subsection{Phase transition of limiting distribution of
  $\lambda_1(\mathbf{S})$ for normal population as to decay rate of equi-correlation
  coefficient}

  The limiting distribution of an appropriate scaling of a translation of the largest eigenvalue $\lambda_1(\mathbf{S})$ of a sample
  covariance matrix $\mathbf{S}$ is Tracy-Widom distribution if the
  equi-correlation coefficient $\rho$ is 0, by
  \citep{10.1214/aos/1009210544,Soshnikov2002},  and a normal
  distribution if $\rho$ is a positive constant by
  Theorem~\ref{thm:paraphrased main theorem}. In view of
  Theorem~\ref{thm:yataaoshima}~\eqref{nrho},
  we turn our attention to the case $\lim_{ N\to\infty}N\rho<\infty$.

 First, we recall \cite[Theorem~1.1 and Corollary~1.1]{baik05:_phase}.
	   
  \begin{proposition}Assume $Q>1$ and the population is
   $\mathrm{N}_N(\bm{\mu},\, {\mathbf{\Sigma}})$, and  $\lambda_1\ge
	   \lambda_2\ge\lambda_{r}\ge\lambda_{r+1}=\cdots= \lambda_N=1$ be the
	   eigenvalues of ${\mathbf{\Sigma}}$. 

\begin{itemize}
\item[(a)] Suppose  $0\leq k\leq r$, 
$\lambda_1 =\cdots= \lambda_k = 1 + 1/\sqrt{Q}$, and $\set{\lambda_{k+1},\ldots,\lambda_r}\subset (0,\, 1+ 1/\sqrt{Q})$. Then,
 \begin{align}
  &\lambda_1(\mathbf{S})\mathrel{\stackrel{P}{\to}} {(1+\sqrt{1/Q})^2},\ \mbox{and}\nonumber
\\
& \Prb\left(\left(\lambda_1(\mathbf{S})-{(1+\sqrt{1/Q})^2}\right)\cdot\frac{\sqrt{Q}}{(1+
 \sqrt{Q})^{4/3}} T^{2/3}\leq x\right)\to F_{\mathrm{TW},k}(x),\label{asymptotic normality}
 \end{align}
where $F_{\mathrm{TW},k}(x)$ is the distribution function of a generalization of the Tracy-Widom distribution~\cite[Definition~1.1]{baik05:_phase}.

\item[(b)]
 Suppose $1\leq k\leq r$, $\lambda_1 =\cdots= \lambda_k > 1 + 1/\sqrt{Q}$, and
$\set{\lambda_{k+1},\ldots,\lambda_r}\subset
	  (0,\lambda_1)$. Then,
\begin{align}
&\lambda_1(\mathbf{S})\stackrel{P}{\to}\lambda_1+\frac{Q^{-1}\lambda_1}{\lambda_1-1},\ \mbox{and}\nonumber
 \\
&
 \label{TW}\Prb\left(\left(\lambda_1(\mathbf{S})-\left(\lambda_1+\frac{Q^{-1} \lambda_1}{\lambda_1-1}\right)\right)\cdot\frac{\sqrt{T}}{\sqrt{\lambda_1^2-Q^{-1} \lambda_1^2/(\lambda_1-1)^2}}\leq x\right)\to F_{\mathrm{G},k}(x),
\end{align}
where $F_{\mathrm{G},k}(x)$ is the distribution function of a generalization of the Gaussian distribution~\cite[Definition~1.2]{baik05:_phase}.
\end{itemize}
\end{proposition}

This proposition suggests the following phase transition:  For a normal population, if $\lim_{ N\to\infty}
N\rho>1/\sqrt{Q}$, then the largest eigenvalue $\lambda_1(\mathbf{S})$
satisfies the asymptotic normality in a sense of \eqref{asymptotic normality}, and if $\lim_{ N\to\infty}  N\rho<1/\sqrt{Q}$,
then an appropriate scaling of a translation of $\lambda_1(\mathbf{S})$ converges to
the Tracy-Widom distribution $G_1$ in a sense of \eqref{TW}. This holds
for $\lim_{ N\to\infty}
N\rho=\infty$ by Theorem~\ref{thm:yataaoshima}.

The detail of this phase transition will be studied elsewhere.

\section{Proof of Theorem~\ref{thm:main}}\label{proof:main}
We also consider centered~(noncentered, resp.) variant of  the sample
covariance matrix $\mathbf{S}\in{{\mathbb R}}^{N\times N}$~(the sample correlation
matrix $\mathbf{C}\in{{\mathbb R}}^{N\times N}$, resp.):
\begin{definition}\label{def:aux}
   Let ${\mathbf{\tilde Y}}$ be
 $\diag\left(\norm{\bm{x}_1}^{-1},\ldots,\norm{\bm{x}_N}^{-1}\right){\mathbf{X}}\in{{\mathbb R}}^{N\times T}$. Define
 the following three matrices:
   \begin{align*}
{\mathbf{\tilde C}}\ \ &:={\mathbf{\tilde Y}}{\mathbf{\tilde Y}}^\top\in{{\mathbb R}}^{N\times N}&\mbox{(\emph{noncentered sample
   correlation matrix})},
   \\
    \mathbf{E}\ \ &:= T^{-1/2}({\mathbf{X}}-\overline{\mathbf{X}})\in {{\mathbb R}}^{N\times  T},
    \\
\mathbf{S}^\circ\ &:=\mathbf{E}\mathbf{E}^\top\in{{\mathbb R}}^{N\times N}  &\mbox{(\emph{centered
sample covariance matrix})}.
   \end{align*}
\end{definition}

To prove Theorem~\ref{thm:main}~\eqref{consistency}, we will prove:
 \begin{theorem}\label{thm:paraphrased main theorem} Assume
  that $\bm{X}_1,\ldots,\bm{X}_T
\mathrel{\stackrel{\mbox{\emph{i.\@i.\@d.\@}}}{\sim}} \mathrm{N}_N (\bm{\mu},\,\bm{\Delta}{\mathbf{\Sigma}}_\rho\bm{\Delta})$,
 $\bm{\mu}\in{{\mathbb R}}^N$ is deterministic, $\bm{\Delta}$ is deterministic,
 nonsingular diagonal, and $\rho\in[0,\,1)$ is deterministic. 
 Suppose $({\mathbf{P}}=\mathbf{S}$, $\bm{\mu}=\bm{0}$, $\bm{\Delta}=\mathbf{I})$, $({\mathbf{P}}=\mathbf{S}^\circ$, $\bm{\Delta}=\mathbf{I})$, $({\mathbf{P}}={\mathbf{\tilde C}}$,
   $\bm{\mu}=\bm{0})$, or ${\mathbf{P}}=\mathbf{C}$. Then,
${\lambda_1({\mathbf{P}})}/{N}\mathrel{\stackrel{a.s.}{\to}}\rho$.
\end{theorem}

\begin{proof}
 First consider the case where ${\mathbf{P}}=\mathbf{S}$, $\bm{\mu}=\bm{0}$, and
$\bm{\Delta}=\mathbf{I}$.  Suppose $\rho=0$. Then, all the entries of ${\mathbf{X}}$ are
standard normal, so they are independent and $\Exp
\abs{x_{11}}^4<\infty$. Thus,
\begin{align*}
 \lim_{N,T\to\infty,{T/N \to Q}}\lambda_1(\mathbf{S})=b_Q\quad (a.s.)
\end{align*} by
\citep[Theorem~2]{bai-yin2}. Hence, Theorem~\ref{thm:paraphrased main
theorem}~\eqref{consistency} holds for this case when $\rho=0$.

Then, we prove $\lambda_1(\mathbf{S})/N\mathrel{\stackrel{a.s.}{\to}}\rho$ for $\rho>0$. For this, we
 recall
 \cite[Proposition~2.1]{Merlev_de_2019}, which is due to~\cite[Proposition~7.3]{Ledoit2018OptimalEO} and
\cite[Theorem~2.1]{https://doi.org/10.48550/arxiv.1711.00217} according
to \cite{Merlev_de_2019}.
Following \cite[p.~8]{Van98}, we say a sequence  $(F_N)_N$ of distribution functions is \emph{tight} on ${{\mathbb R}}^+$, if
for every $\epsilon\in(0,\,1]$, there is a constant $M\ge0$ such that
\[
 \sup_N F_N(M)>1-\epsilon.
\]
  \begin{proposition}[\protect{\citet[Proposition~2.1]{Merlev_de_2019}}]\label{prop:ledoit wolf}
  Assume
  \begin{enumerate}
   \item \label{A1}
$ \mathbf{S}=T^{-1}{{\mathbf{\Sigma}}}^{1/2}{\mathbf{G}}
	 {\mathbf{G}}^\top{{\mathbf{\Sigma}}}^{1/2}\in{{\mathbb R}}^{N\times N}$, the entries of
  ${\mathbf{G}}\in{{\mathbb R}}^{N \times T}$ are independent, standard normal,
and ${{\mathbf{\Sigma}}}\in{{\mathbb R}}^{N\times N}$ is positive definite, 
	 symmetric, and deterministic.
   \item \label{A2}
A sequence	 
$\left(F^{{{\mathbf{\Sigma}}}}\right)_N$ of empirical spectral distributions
is tight on ${{\mathbb R}}^+$,
	 and $\lim_{ N\to\infty}\lambda_1({{\mathbf{\Sigma}}})=\infty$.
   \item  ${{N,T\to\infty,\ {T/N \to Q}}}\in(0,\infty)$. 
  \end{enumerate}
   Then,
${\lambda_1(\mathbf{S})}/{\lambda_1({{\mathbf{\Sigma}}})}\mathrel{\stackrel{a.s.}{\to}} 1.$
  \end{proposition}
The sequence $\left(F^{{\mathbf{\Sigma}}_\rho}\right)_N$
of the empirical spectral distribution function of ${\mathbf{\Sigma}}_\rho\in{{\mathbb R}}^{N\times N}$ is a tight sequence on
${{\mathbb R}}^+$. To see it,
consider $\epsilon\in (0,\,1]$ and $M=(\epsilon^{-1}-1)\rho+1>0$. For
every integer $N>M$, $\lambda_1({\mathbf{\Sigma}}_\rho)> M$, so
$F^{{\mathbf{\Sigma}}_\rho}(M)=1$.  Hence, $\sup_N F^{{\mathbf{\Sigma}}_\rho}(M)\ge
1-\epsilon$.  Therefore, we have the tightness of the sequence.

Then, by Proposition~\ref{prop:ledoit wolf} with ${{\mathbf{\Sigma}}}:={\mathbf{\Sigma}}_\rho\in{{\mathbb R}}^{N\times N}$,
${\lambda_1(\mathbf{S})}\big/\left((N-1)\rho +
1\right)={\lambda_1(\mathbf{S})}\big/{\lambda_1({\mathbf{\Sigma}}_\rho)} \mathrel{\stackrel{a.s.}{\to}} 1.$
Thus, we have established $\lambda_1(\mathbf{S})/ N\mathrel{\stackrel{a.s.}{\to}}\rho$ for every $\rho\ge0$.

\medskip
Now, we consider the case where ${\mathbf{P}}=\mathbf{S}^\circ$ and
 $\bm{\Delta}=\mathbf{I}$. We will derive ${\lambda_1(\mathbf{S}^\circ)}/{N}\mathrel{\stackrel{a.s.}{\to}}\rho$ for
every $\rho\ge0$, by using the following:
\begin{proposition}[\protect{\citet[Lemma~2]{bai-yin2}}]\label{prop:bai-yin2}
Let $\set{y_{kj} \colon\, k,j\ge1}$
be a double array of \emph{i.\@i.\@d.\@} random variables and let $\alpha>{1}/{2}$, $\beta\ge 0$ and $M>0$ be constants. Then,
\begin{align*}
& \max_{j\le M T^\beta}\abs*{\sum_{k=1}^T  \frac{y_{kj}-m}{T^\alpha}}\mathrel{\stackrel{a.s.}{\to}} 0\quad (T\to\infty)\\
\iff&
\Exp\abs{y_{11}}^{\frac{1+\beta}{\alpha}}< \infty\; \&
\;  m=\begin{cases}\displaystyle
\Exp y_{11},&(\alpha\le 1),\\
\displaystyle\text{\rm any},&(\alpha>1).
\end{cases}
\end{align*}
\end{proposition}

Proposition~\ref{prop:bai-yin2} is also used to prove
Theorem~\ref{thm:equilsd2} in Section~\ref{proof:equilsd2}.

\medskip 
The centered sample covariance matrix
$\mathbf{S}^\circ$ is invariant with respect to the shifting of variables. By
this invariance, we can
assume $\bm{\mu}=\bm{0}.$
It suffices to assure
   \begin{align*}
\lim_{{\substack{N,T\to\infty\\ {T/N \to Q}}}}    \abs*{\sqrt{\frac{\lambda_1(\mathbf{S}^\circ)}{N}}-\sqrt{\frac{\lambda_1(\mathbf{S})}{N}}} = 0\quad(a.s.) .
   \end{align*}

By the triangle inequality for the spectral norm,
 \begin{align}
   \label{ign}
  \abs*{\sqrt{\lambda_1(\mathbf{S}^\circ)} - \sqrt{\lambda_1(\mathbf{S})}} =
\abs*{\snorm{ T^{-1/2}({\mathbf{X}}-\overline{\mathbf{X}})} - \snorm{ T^{-1/2}{\mathbf{X}}}} \le
  \snorm{ T^{-1/2}\overline{\mathbf{X}}}.
 \end{align}

 The last term $\snorm{ T^{-1/2}\overline{\mathbf{X}}}$ is $\sqrt{\sum_{i=1}^N \overline{x}_i^2}$
 (a.s.).
 To see it, note that the $i t$-entry of
 $T^{-1}\overline{\mathbf{X}}\,\overline{\mathbf{X}}^\top$ is $T^{-1}\sum_{k=1}^T  \overline{x}_i\overline{x}_t = \overline{x}_i
 \overline{x}_t$.
 Therefore, $[\overline{x}_1,\ldots,\overline{x}_N]^\top$ is an eigenvector of $T^{-1}\overline{\mathbf{X}}\,\overline{\mathbf{X}}^\top$,
 corresponding to an eigenvalue $\sum_{i=1}^N \overline{x}_i^2$. Clearly, $\rank
 \overline{\mathbf{X}}\le 1$ and almost surely $\rank
 \overline{\mathbf{X}}=1$. Thus, it is also the case for $\rank T^{-1}\overline{\mathbf{X}}\,\overline{\mathbf{X}}^\top$.
Considering the eigenvalue $\sum_{i=1}^N \overline{x}_i^2$ of
 $T^{-1}\overline{\mathbf{X}}\,\overline{\mathbf{X}}^\top$ is positive
 almost surely, the other $N-1$ eigenvalues are 0 almost surely.
 
By dividing \eqref{ign} with $\sqrt{N}$, it holds almost surely that
\begin{align}
\label{qq}   
  \abs*{\sqrt{\frac{\lambda_1(\mathbf{S}^\circ)}{N}} - \sqrt{\frac{\lambda_1(\mathbf{S})}{N}}} \le \sqrt{\sum_{i=1}^N \frac{\overline{x}_i^2}{N}}
 \le\sqrt{\max_{1\le i\le N}\left({\sum_{t=1}^ T}
  \frac{x_{i t}}{T}\right)^2} =\max_{1\le i\le N}\abs*{{\sum_{t=1}^ T}\frac{x_{i t}}{T}}.
\end{align}
By $\bm{\Delta}=\mathbf{I}$ and $\bm{\mu}=\bm{0}$, $x_{i t}$ $(i,t\ge1)$ are standard normal
  random variables equi-correlated with $\rho\in [0,1)$.
By decomposition~\eqref{decomp}, there are independent standard normal
 random variables ${{f}}_t, e_{i t}$ $(i,t\ge1)$ such that
 $x_{i t}=\sqrt{\rho}{{f}}_t +\sqrt{1-\rho}e_{i t}$. By the triangle
  inequality,
  \begin{align*}
   \max_{1\le i\le N}\abs*{{\sum_{t=1}^ T}\frac{x_{i t}}{T}} \le    \abs*{{\sum_{t=1}^ T}\frac{
  \sqrt{\rho}{{f}}_t}{T}} +  \max_{1\le i\le N}\abs*{{\sum_{t=1}^ T}\frac{
  \sqrt{1-\rho}e_{i t}}{T}} .
  \end{align*}
 The first term of the right side tends to 0 in $T\to\infty$ almost
 surely, by the law of large numbers.
 The second term of the right side tends to 0 in $N,T\to\infty, \ {T/N \to Q}$ almost
 surely, by Proposition~\ref{prop:bai-yin2} with $\alpha=\beta=1$ and $M=Q$. Thus,
 \begin{align}
  \label{cas0}\max_{1\le i\le N}\abs{\overline{x}_i}=\max_{1\le i\le N}\abs*{{\sum_{t=1}^ T}\frac{ x_{i t}}{T}} {\stackrel{a.s.}{\to}}0\quad({{N,T\to\infty,\ {T/N \to Q}}}).
 \end{align}
Hence, $\eqref{qq}\mathrel{\stackrel{a.s.}{\to}}0$. Since we have already established
${\lambda_1(\mathbf{S})}/{N}\mathrel{\stackrel{a.s.}{\to}}
\rho$ for every $\rho\ge0$, we then derived
${\lambda_1(\mathbf{S}^\circ)}/{N}\mathrel{\stackrel{a.s.}{\to}} \rho$ for every $\rho\ge0$.

\medskip We consider the case ${\mathbf{P}}={\mathbf{\tilde C}}$ and $\bm\mu=\bm{0}$.  Because the
noncentered sample correlation matrix ${\mathbf{\tilde C}}$ is invariant under the
scaling, we can assume
\begin{align*}
\bm{\Delta}=\mathbf{I}.
\end{align*}
From ${\lambda_1(\mathbf{S})}/{N}\mathrel{\stackrel{a.s.}{\to}} \rho$
 of Theorem~\ref{thm:paraphrased main theorem}, we will guarantee
   \begin{align}
\lim_{{\substack{N,T\to\infty\\ {T/N \to Q}}}}    \abs*{\sqrt{\frac{\lambda_1({\mathbf{\tilde C}})}{N}}-\sqrt{\frac{\lambda_1(\mathbf{S})}{N}}} = 0\quad(a.s.) .\label{tCS:cas0}
   \end{align}
For a diagonal matrix
\begin{align*}
{{\tilde{\mathbf{D}}}}=\diag\left(\frac{\sqrt{T}}{\norm{{\bm{x}_1}}},\ldots,\frac{\sqrt{T}}{\norm{{\bm{x}_N}}}\right),\quad
 \norm{\bm{x}_i}=\sqrt{{\sum_{t=1}^ T} x_{it}^2} 
\end{align*}
we have ${\mathbf{\tilde Y}}= T^{-1/2}{{\tilde{\mathbf{D}}}} {\mathbf{X}}$.

Recall that
  the \emph{spectral norm} of a matrix $\mathbf{M}\in{{\mathbb R}}^{m\times n}$ is
$$\snorm{\mathbf{M}}=\sup_{\bm{X}\in{{\mathbb R}}^n,\norm{\bm{X}}=1} \norm{\mathbf{M} \bm{X}}.$$ Note
$\snorm{\mathbf{M}}=\sqrt{\lambda_1(\mathbf{M}^\top\mathbf{M})}$.
By the triangle inequality of the spectral norm and that
$\snorm{\mathbf{M}_1\mathbf{M}_2} \le \snorm{\mathbf{M}_1} \snorm{\mathbf{M}_2}$ for every $\mathbf{M}_1\in
 {{\mathbb R}}^{N\times N}$ and
 $\mathbf{M}_2\in{{\mathbb R}}^{N\times T}$, 
   \begin{align}
&  \abs*{  \sqrt{\frac{\lambda_1({\mathbf{\tilde C}})}{N}}-\sqrt{\frac{\lambda_1(\mathbf{S})}{N}}}
 \le  \snorm*{\frac{1}{\sqrt{N}}{\mathbf{\tilde Y}}- \frac{1}{\sqrt{N T}}
   {\mathbf{X}}}=\snorm{(N T)^{-1/2}({{\tilde{\mathbf{D}}}} - \mathbf{I}){\mathbf{X}}}\\
 &
 \le\snorm{{{\tilde{\mathbf{D}}}}-\mathbf{I}}\cdot\snorm{(N T)^{-1/2}{\mathbf{X}}}=\max_{{1\le i\le N}}\abs*{\frac{ T^{1/2}}{\norm{\bm{x}_i}}
 - 1 }\cdot \sqrt{\frac{\lambda_1(\mathbf{S})}{N}}.
\label{the same argument}
  \end{align}
   Since we established ${\lambda_1(\mathbf{S})}/{N}\mathrel{\stackrel{a.s.}{\to}}\rho$ from $\bm{\Delta}=\mathbf{I}$ and $\bm{\mu}=\bm{0}$,
   it suffices to verify
\begin{align}
\lim_{{\substack{N,T\to\infty\\ {T/N \to Q}}}}   \max_{{{1\le i\le N}}}\abs*{\frac{\sqrt T}{\norm{\bm{x}_i}} -1}=0\quad(a.s.). \label{uvw}
\end{align}
   By decomposition~\eqref{decomp}, there are independent standard
 normal random variables ${{f}}_t, e_{i t}$ $(i,t\ge1)$ such
 that $x_{i t}=\sqrt{\rho}{{f}}_t +\sqrt{1-\rho}e_{i t}$. Hence,
\begin{align*}{\sum_{t=1}^ T} \frac{x_{i t}^2}{T}= \rho
 {\sum_{t=1}^ T}\frac{{{f}}_t^2}{T}+2\sqrt{\rho(1-\rho)}{\sum_{t=1}^ T}\frac{{{f}}_t e_{i t}}{T}+(1-\rho){\sum_{t=1}^ T}\frac{ e_{i t}^2}{T}.
\end{align*}
Then 
 \begin{align*}
&\max_{{{1\le i\le N}}}\abs*{\frac{\norm{\bm{x}_i}^2}{T} - 1} =
  \max_{{{1\le i\le N}}}\abs*{{\sum_{t=1}^ T}\frac{x_{i t}^2}{T} - 1}\\
\le
 & \rho\max_{{{1\le i\le N}}}\abs*{{\sum_{t=1}^ T}\frac{{{f}}_t^2}{T} - 1}
  + 2\sqrt{\rho(1-\rho)}\max_{{1\le i\le N}}\abs*{{\sum_{t=1}^ T}\frac{{{f}}_t e_{i t}}{T}}
 + (1-\rho)\max_{{{1\le i\le N}}}\abs*{{\sum_{t=1}^ T}\frac{ e_{i t}^2}{T} - 1 }.
 \end{align*}
The right side converges almost surely to 0 by Proposition~\ref{prop:bai-yin2},
 because 
 ${{f}}_t^2$
 $(t\ge1)$ are i.\@i.\@d.\@ of unit mean,  ${{f}}_t  e_{i t}$
 $(i,t\ge1)$ are centered i.\@i.\@d.\@, and $ e_{i t}^2$ $(i,t\ge1)$ are
 i.\@i.\@d.\@ of unit mean.  Hence,
\begin{align}
\label{ato}
\lim_{{\substack{N,T\to\infty\\ {T/N \to Q}}}}\max_{{{1\le i\le N}}}\abs*{\frac{\norm{\bm{x}_i}^2}{T} - 1}=0\quad(a.s.),
\end{align}
from which $\abs*{{\norm{\bm{x}_i}^2}/{T} - 1}>
\abs*{{\norm{\bm{x}_i}}/{\sqrt T} - 1}$ implies
\begin{align}
\label{at}
\lim_{{\substack{N,T\to\infty\\ {T/N \to Q}}}}\max_{{{1\le i\le N}}}\abs*{\frac{\norm{\bm{x}_i}}{\sqrt T} - 1}=0\quad(a.s.).
\end{align}
Suppose $\max_{1\le i\le N}\abs*{{\norm{\bm{x}_i}}/{\sqrt T} - 1}<\epsilon$ for sufficiently
small $\epsilon>0$. Then, for all $i$ $(1\le i\le N)$,
$1-\epsilon < {\norm{\bm{x}_i}}/\sqrt{T}  < 1+\epsilon$,  which implies
$(1+\epsilon)^{-1} -1 <\sqrt{T}/{\norm{\bm{x}_i}} -1 <
(1-\epsilon)^{-1}-1$. Hence, for all $i$ $(1\le i\le N)$, $\abs*{\sqrt{T}/{\norm{\bm{x}_i}} -1} <
\max((1-\epsilon)^{-1}-1,\ \abs*{(1+\epsilon)^{-1}-1})<
\epsilon+\epsilon^2+\epsilon^3+\cdots <2\epsilon$, because $0<\epsilon\ll1$. Thus,
$\max_{1\le i\le N}\abs*{\sqrt{T}/{\norm{\bm{x}_i}} -1}< 2\epsilon$.
Therefore, from \eqref{at},  \eqref{uvw} follows.

Thus, \eqref{tCS:cas0} is guaranteed. Therefore,  ${\lambda_1({\mathbf{\tilde C}})}/{N}\mathrel{\stackrel{a.s.}{\to}}\rho$
  follows because the assumption $\bm{\mu}=\bm{0}$ implies
  ${\lambda_1(\mathbf{S})}/{N}\mathrel{\stackrel{a.s.}{\to}} \rho$.
  
\bigskip
At last, we will demonstrate ${\lambda_1(\mathbf{C})}/{N}\mathrel{\stackrel{a.s.}{\to}} \rho$ of Theorem~\ref{thm:paraphrased main theorem}.
Because the sample correlation matrix $\mathbf{C}$ is invariant under the shifting and scaling, we can assume
$\bm{\Delta}=\mathbf{I}$ and $\bm{\mu}=\bm{0}$. For $\mathbf{C}={\mathbf{Y}}{\mathbf{Y}}^\top$,
we have ${\mathbf{Y}} =  T^{-1/2}\mathbf{D}({\mathbf{X}} - \overline{\mathbf{X}})$ where
\begin{align*} \mathbf{D} = \diag\left( {\sqrt{T}}/{\norm{{\bm{x}_1 -{\overline{\bm{x}}}_1}}},\ldots,{\sqrt{T}}/{\norm{{\bm{x}_N -{\overline{\bm{x}}}_N}}}\right).\end{align*}

Recall $\mathbf{S}^\circ=T^{-1}({\mathbf{X}} -\overline{\mathbf{X}})({\mathbf{X}} -\overline{\mathbf{X}})^\top$. By the same arguments as in \eqref{the same argument},
 \begin{align*}
\abs*{\sqrt{\frac{\lambda_1(\mathbf{C})}{N}}
  -\sqrt{\frac{\lambda_1(\mathbf{S}^\circ)}{N}}} \le
  \max_{{1\le i\le N}}\abs*{\frac{\sqrt{T}}{\norm{{\bm{x}_i-{\overline{\bm{x}}}_i}}} -1} \cdot \sqrt{\frac{\lambda_1(\mathbf{S}^\circ)}{N}} .
 \end{align*}
Because ${\lambda_1(\mathbf{S}^\circ)}/{N}\mathrel{\stackrel{a.s.}{\to}}\rho$ follows from $\bm{\Delta}=\mathbf{I}$ we
 assumed, it suffices to show
\begin{align}\lim_{{\substack{N,T\to\infty\\ {T/N \to Q}}}} \max_{{1\le i\le N}}\abs*{\frac{\sqrt{T}}{\norm{{\bm{x}_i-{\overline{\bm{x}}}_i}}} -1}=0\quad(a.s.).
\label{qqq}\end{align} By $\norm{\bm{x}_i -{\overline{\bm{x}}}_i}^2 = \norm{\bm{x}_i}^2 - T\abs*{\overline{x}_i}^2$, 
 \begin{align*}
\max_{{{1\le i\le N}}}\abs*{\frac{\norm{\bm{x}_i-{\overline{\bm{x}}}_i}^2}{T}-1}\le \max_{{{1\le i\le N}}}\abs*{\frac{\norm{\bm{x}_i}^2}{T}-1}+\max_{{{1\le i\le N}}}\abs*{\overline{x}_i}^2.
 \end{align*}
 The first~(second, resp.) term of the right side converges almost surely to 0, by
\eqref{ato}~(\eqref{cas0}, resp.). In a similar way as above,
\eqref{qqq} is shown.
Therefore,  ${\lambda_1(\mathbf{C})}/{N}\mathrel{\stackrel{a.s.}{\to}} \rho$ for every $\rho\ge0$.
This completes the proof of Theorem~\ref{thm:paraphrased main theorem}.

\end{proof}

\medskip
To prove Theorem~\ref{thm:main}~\eqref{fluctuation}, we employ:
\begin{proposition}[\protect{\citet[Theorem~2.2]{Merlev_de_2019}}]\label{Merlevde19:Theorem
 2.2}
 Assume the three assumptions of Proposition~\ref{prop:ledoit
 wolf}. From a \emph{spectral gap condition} on ${{\mathbf{\Sigma}}}$:
   \begin{align}   \label{spectral gap condition}
    \limsup_{{\substack{N,T\to\infty\\ {T/N \to Q}}}}\frac{\lambda_2({{\mathbf{\Sigma}}})}{\lambda_1({{\mathbf{\Sigma}}})}
< 1,
  \end{align}
it follows that
\begin{align*}
\sqrt{T}\left(\frac{\lambda_1(\mathbf{S})}{\lambda_1({{\mathbf{\Sigma}}})} -1 -
\frac{1}{T}\sum_{k=2}^N \frac{\lambda_k({{\mathbf{\Sigma}}})}{\lambda_1({\mathbf{\Sigma}})
 - \lambda_k({\mathbf{\Sigma}})}\right)\mathrel{\stackrel{D}{\to}} \mathrm{N}(0,\ 
2).
\end{align*}
 \end{proposition}
The spectral gap
condition~\eqref{spectral gap condition} holds for
${\mathbf{\Sigma}}={\mathbf{\Sigma}}_\rho$ with $\rho>0$, because \eqref{ev}
implies
\[
  \limsup_{{\substack{N,T\to\infty\\ {T/N \to Q}}}}\frac{\lambda_2({\mathbf{\Sigma}})}{\lambda_1({{\mathbf{\Sigma}}})}=\limsup_{{\substack{N,T\to\infty\\ {T/N \to Q}}}} \frac{1-\rho}{(N-1)\rho+1}=0.
\]
Then,
${T}^{-1}\sum_{k=2}^ N{\lambda_k({{\mathbf{\Sigma}}})}/{\left(\lambda_1({{\mathbf{\Sigma}}})
- \lambda_k({{\mathbf{\Sigma}}})\right)} = {{(N -1)(1-\rho)}/{(N T\rho)}}$. This completes the proof of Theorem~\ref{thm:main}~\eqref{fluctuation}.

\section{Proof of Corollary~\ref{cor:soundness}}\label{proof:soundness}

From Proposition~\ref{prop:AH22}, $F^{\mathbf{C}}(x)$ converges to
 $\mathrm{MP}_{Q,\, 1-\rho}(x)$ almost surely for every $x\ne0$. It suffices to prove
that
 $\mathrm{MP}_{T/N ,\, 1-\lambda_1(\mathbf{C})/N}(x)$ converges almost surely to
 $\mathrm{MP}_{Q,\,1-\rho}(x)$ for any $x\ne0$. By Theorem~\ref{thm:main},
 $\lambda_1(\mathbf{C})/N\mathrel{\stackrel{a.s.}{\to}} \rho$, so $x/(1-\lambda_1(\mathbf{C})/N)) \mathrel{\stackrel{a.s.}{\to}}
 x/(1-\rho)>0$.  The distribution function $\mathrm{MP}_{Q}(t)$ of a
 Mar\v{c}enko-Pastur distribution is continuous with respect to
  $t>0$. Thus,
\begin{align}
  \label{rhoconv}
  \lim_{{\substack{N,T\to\infty\\ {T/N \to Q}}}} \abs*{\mathrm{MP}_{T/N }\left(\frac{x}{1-\lambda_1(\mathbf{C})/N}\right)-
 \mathrm{MP}_{T/N }\left(\frac{x}{1-\rho}\right)}=0\quad(a.s.).
\end{align}
On the other hand, $\mathrm{MP}_{Q}(t)$ is continuous
 with respect to $Q\in(0,\infty)$ for each $t>0$, in view of the density
 function~\eqref{density}. Therefore,
\begin{align*}
   \lim_{{\substack{N,T\to\infty\\ {T/N \to Q}}}}
\abs*{\mathrm{MP}_{T/N }\left(\frac{x}{1-\rho}\right)-
  \mathrm{MP}_{Q}\left(\frac{x}{1-\rho}\right)}=0\quad(a.s). 
\end{align*}
By this, \eqref{rhoconv}, and the triangle inequality,
\begin{align*}
\lim_{{\substack{N,T\to\infty\\ {T/N \to Q}}}} \abs*{\mathrm{MP}_{T/N }\left(\frac{x}{1-\lambda_1(\mathbf{C})/N}\right)-
 \mathrm{MP}_{Q}\left(\frac{x}{1-\rho}\right)}
  = 0\quad(a.s.).
\end{align*}
\section{Proof of Theorem~\ref{thm:equilsd2}}\label{proof:equilsd2}

\subsection{LSD of sample covariance matrix generated from factor model}
  
In this subsection, we provide the LSD of the sample covariance matrix
  $\mathbf{S}$ generated from a factor model~(Definition~\ref{def:factor model})
  in $N,T\to\infty,{T/N \to Q}$ with Assumption~\ref{ass:conv}. Let $F_1$ and $F_2$ be distribution functions.
The \emph{L{\'e}vy distance} between $F_1$ and
$F_2$ is denoted by $\Le(F_1,F_2)$. For the definition, see
\cite[Definition~2.7]{huber09:_robus}. By \cite[Theorem~2.9]{huber09:_robus},
\begin{proposition}
 \label{metrize}
  For every sequence $(F_n)_n$ of distribution functions, and every
distribution function $F$, $F_n$ converges weakly to $F$
 if and only if $\Le(F_n,F)\to0$.
\end{proposition}
By \citet[p.~36]{huber09:_robus}, the \emph{Kolmogorov distance} between two $F_1$ and $F_2$ is defined as
$
\K(F_1,F_2)=\sup_ {x\in {{\mathbb R}}}\abs*{F_1(x)-F_2(x)}$, 
and satisfies
\begin{align}
 \Le(F_1,F_2)\le \K(F_1,F_2).\label{huber2}
\end{align}

\begin{proposition}[\protect{\citet[Lemma~2.6, the rank inequality]{bai99:_method}}] \label{rank}
\begin{align*}
\K\left(F^{{\mathbf{A}}{\mathbf{A}}^\top},F^{\mathbf{B}\mathbf{B}^\top}\right)\leq\frac{1}{N}\;\rank({\mathbf{A}}-\mathbf{B}),\qquad\left({\mathbf{A}},\mathbf{B}\in{{\mathbb R}}^{N\times T}\right).\end{align*}
\end{proposition}

\begin{theorem}[Scaling]\label{thm:equilsd}Given a factor model.
 Suppose ${{N,T\to\infty,\ {T/N \to Q}}}\in(0,\,\infty)$ and
 Assumption~\ref{ass:conv}. Then, the following two hold almost surely:
 \begin{enumerate}
  \item \label{assert:S} $F^{\mathbf{S}}$  converges weakly
 almost surely to
$\mathrm{MP}_{Q,\sigma_\infty^2(1-\rho)}$, if $\mu_i=0$ $(1\le i\le N)$. 
\item \label{assert:EE} $F^{\mathbf{E}\mathbf{E}^\top}$ converges weakly
 almost surely to
$\mathrm{MP}_{Q,\sigma_\infty^2(1-\rho)}$. \end{enumerate}
\end{theorem}
  \begin{proof}By Definition~\ref{def:factor model},
   the data matrix
   is written as
    \begin{align}
    {\mathbf{X}}={\sum_{k=1}^K} \left[\ell_{ik} f_{t k}\right]_{i t}+\bm{\Psi},\quad  \bm{\Psi}:=[e_{i t}]_{i t}\in{{\mathbb R}}^{N\times T}\label{xwx}.  \end{align} 
By Assumption~\ref{ass:conv} and Proposition~\ref{prop:MP law}~\eqref{assert:MP LT}, the ESD of $T^{-1} \bm{\Psi} \bm{\Psi}^\top$ converges weakly to $\mathrm{MP}_{Q,\psi_1}$ almost surely. Then, Proposition~\ref{metrize} implies
\begin{align}\label{xi}
\lim_{{\substack{N,T\to\infty\\ {T/N \to Q}}}} \Le\left(F^{{T}^{-1}\bm{\Psi} \bm{\Psi}^\top},\,
\mathrm{MP}_{Q,\psi_1}
 \right)=0\quad(a.s.)
\end{align}
where $\psi_1=\sigma_\infty^2(1-\rho)$ by Assumption~\ref{ass:conv}.
Recall  $\mathbf{S}=T^{-1}{\mathbf{X}}{\mathbf{X}}^\top$.
Thus, by \eqref{huber2}, Proposition~\ref{rank}, \eqref{xwx}, and
  $\rank[\ell_{ik} f_{t k}]_{i t}\le 1$ $(1\le k\le K)$, it holds
   almost surely that
\begin{align*}
&\Le\left(F^{\mathbf{S}},\ 
F^{{T}^{-1}\bm{\Psi} \bm{\Psi}^\top}\right)\le
 \K\left(F^{\mathbf{S}},\ F^{{T}^{-1}\bm{\Psi} \bm{\Psi}^\top}\right)
\le
 \frac{1}{N}\rank\left(\frac{1}{\sqrt{T}}{\mathbf{X}}-\frac{1}{\sqrt{T}}\bm{\Psi}\right)
 \\
&=  \frac{1}{N}\rank
 \left(\frac{1}{\sqrt{T}}{\sum_{k=1}^K}\left[\ell_{ik} f_{t k}\right]_{i t}
 \right)
\le  \frac{1}{N}{\sum_{k=1}^K}\rank
 \left(\left[\ell_{ik} f_{t k}\right]_{i t}
 \right)
\le \frac{K}{N}\to 0 \qquad( N\to \infty).
\end{align*}
By this, \eqref{xi}, and the triangle inequality,
\begin{align}\label{FS}
 \lim_{{\substack{N,T\to\infty\\ {T/N \to Q}}}} \Le\left(F^{\mathbf{S}},\ 
\mathrm{MP}_{Q,\psi_1}
\right)= 0\quad(a.s.),
\end{align}
which implies the first assertion of this theorem through
   Proposition~\ref{metrize}. By the same proposition,
   the second assertion
   of this theorem follows by the triangle inequality from \eqref{FS} and $$\lim_{{\substack{N,T\to\infty\\ {T/N \to Q}}}} \Le\left(F^{\mathbf{E}\mathbf{E}^\top},\
  \mathrm{MP}_{Q,\psi_1}
  \right)=0\quad (a.s.).$$ The last is because
by \eqref{huber2}, Proposition~\ref{rank}, and $\rank (\overline{\mathbf{X}})\le 1$,
it holds almost surely that
\begin{align*}
\Le(F^{\mathbf{E}\mathbf{E}^\top},\ F^{\mathbf{S}})\le
 \K\left(F^{\mathbf{E}\mathbf{E}^\top},\ F^{{T}^{-1}{{\mathbf{X}}{\mathbf{X}}^\top}}\right)
\le \frac{\rank \left(\overline{\mathbf{X}}\right)}{N}
\le {1}/{N}\to 0\quad( N\to \infty).
   \end{align*}
This completes the proof of Theorem~\ref{thm:equilsd}.  \end{proof}

\subsection{LSD of sample correlation matrix generated from factor model}

In this subsection, we provide the LSD of $\mathbf{C}$ generated from a factor
model~(Definition~\ref{def:factor model}) in $N,T\to\infty,{T/N \to Q}$ with
Assumption~\ref{ass:conv}.
By Assumption~\ref{ass:conv},
\begin{align}L:=\sup\set{
\abs*{\ell_{ik}}\colon {1\le i,\ 1\le k\le K}}<\infty. \label{L1norm}\end{align}

  For $\mu_i=0$ $(1\le i\le N)$,
\begin{align}
 \label{prel}
{\sum_{t=1}^ T}\frac{x_{i t}^2}{T}= \frac{\Tr\left(  \bm{\ell}_i^\top{\bm{\ell}}_i {\sum_{t=1}^ T}
  {{\bm{f}}}_t{{\bm{f}}}_t^\top \right)}{T}
  +2{\sum_{t=1}^ T}\frac{{\bm{\ell}}_i{{\bm{f}}}_t e_{i t}}{T}+{\sum_{t=1}^ T}\frac{ e_{i t}^2}{T}
\end{align}
 because $$\frac{1}{T}{\sum_{t=1}^ T}{{\bm{f}}}_t^\top {\bm{\ell}}_i^\top{\bm{\ell}}_i {{\bm{f}}}_t=\Tr\left( {\bm{\ell}}_i^\top{\bm{\ell}}_i \frac{{\sum_{t=1}^ T}
{{\bm{f}}}_t{{\bm{f}}}_t^\top}{T} \right).$$

\begin{lemma}\label{lem:lim2}For a factor model with $\mu_i=0$ $(1\le i\le N)$, if
 Assumption~\ref{ass:conv}, then
 \begin{align*}
  \lim_{{{\substack{N,T\to\infty\\ {T/N \to Q}}}}}\frac{1}{N}\Tr\mathbf{S}= \sigma_\infty^2\quad (a.s.).
 \end{align*}
\end{lemma}

\begin{proof}
By~\eqref{prel}, $\Tr \mathbf{S}/N=\sum_{i=1}^N {\sum_{t=1}^ T}
 {x_{i t}^2}/{(N T)}$ is equal to
\begin{align}\label{lem:weak}
\Tr\left(\frac{\sum_{i=1}^N {\bm{\ell}}_i^\top{\bm{\ell}}_i}{N} \frac{{\sum_{t=1}^ T}
{{\bm{f}}}_t{{\bm{f}}}_t^\top}{T} \right)
 +2\frac{\sum_{i=1}^N {\sum_{t=1}^ T}{\bm{\ell}}_i {{\bm{f}}}_t e_{i t}}{N T}
 +\frac{\sum_{i=1}^N {\sum_{t=1}^ T} e_{i t}^2}{N T}.
\end{align}
 
We will compute the almost sure limits of the three terms of
 \eqref{lem:weak}, in $N,T\to\infty,{T/N \to Q}$.

By  the continuous mapping
 theorem~\cite[Theorem~2.3]{Van98} applied to the continuous function
 $\Tr$, the limit of the first term of \eqref{lem:weak} is
\begin{align}
    \Tr\left( \lim_{{{\substack{N,T\to\infty\\ {T/N \to Q}}}}}\frac{\sum_{i=1}^N {\bm{\ell}}_i^\top{\bm{\ell}}_i}{N} \frac{{\sum_{t=1}^ T}
   {{\bm{f}}}_t{{\bm{f}}}_t^\top}{T}\right)=
    \Tr\left( \lim_{ N\to\infty}\frac{\sum_{i=1}^N {\bm{\ell}}_i^\top{\bm{\ell}}_i}{N} \lim_{T\to\infty}\frac{{\sum_{t=1}^ T}
   {{\bm{f}}}_t{{\bm{f}}}_t^\top}{T}\right).
\end{align}
By Assumption~\ref{ass:conv},
 $\lim_{i\to\infty}{\bm{\ell}}_i^\top{\bm{\ell}}_i=\bm{\ell}^\top\bm{\ell}$,
 so the Ces\`aro sum~\citep{abbott15:_under_analy} satisfies
$\lim_{ N\to\infty}{\sum_{i=1}^N 
  {\bm{\ell}}_i^\top{\bm{\ell}}_i}/{N}  =  \bm{\ell}^\top\bm{\ell}.$
On other hand, by Definition~\ref{def:factor model}, $f_{t k}$ $( 1\le t\le  T,\ 1\le
 k\le K)$ are centered i.i.d. Thus, $f_{t k}f_{t l}$ $(t=1,2,\ldots)$ are i.i.d., for each
 $k,l\in\{1,\ldots,K\}$. Since $\Exp(f_{t k}f_{t l})=1\ (k=l); 0\ (k\ne l)$,  the strong law of large numbers implies that
\begin{align}
 \lim_{T\to\infty} \frac{{\sum_{t=1}^ T}
 {{\bm{f}}}_t{{\bm{f}}}_t^\top}{T}=\Exp {{\bm{f}}}_t{{\bm{f}}}_t^\top =
 \mathbf{I}_K\quad(a.s.).\label{uv}
\end{align}
By this, 
  the limit of the first term of \eqref{lem:weak} is:
\begin{align}\label{p0}
\lim_{{{\substack{N,T\to\infty\\ {T/N \to Q}}}}} \Tr\left(\frac{\sum_{i=1}^N {\bm{\ell}}_i^\top{\bm{\ell}}_i}{N} \frac{{\sum_{t=1}^ T}
{{\bm{f}}}_t{{\bm{f}}}_t^\top}{T} \right) = \Tr \left(
 \bm{\ell}^\top\bm{\ell} \mathbf{I}_K\right)=\norm{\bm{\ell}}^2\quad(a.s.).
\end{align}
 
Next, we consider the second term of \eqref{lem:weak}. We will verify:
\begin{align}
 \label{snd} \lim_{{\substack{N,T\to\infty\\ {T/N \to Q}}}} 2\frac{\sum_{i=1}^N 
{\sum_{t=1}^ T}{\bm{\ell}}_i{{\bm{f}}}_{t} e_{i t}}{N T}=0\quad(a.s.).
\end{align}
Let $(i_j,\  t_j)$ $(j=1,2,3,\ldots)$ be an
enumeration of $\set{(i,t) \colon\, i,t\in\mathbb{Z}_{\ge1}}$. Because
$f_{ t_j,k}$ $(j\ge1,\ 1\le k\le K)$ and $ e_{i_j, t_j}$ $(j\ge1)$
are independent and centered, so are
${\bm{\ell}}_{i_j}{{\bm{f}}}_{ t_j}$ and $ e_{i_j, t_j}$.  Then,
$\sum_{j=1}^\infty j^{-2}\var\left({\bm{\ell}}_{i_j}
{{\bm{f}}}_{ t_j} e_{i_j,  t_j}\right)=\sum_{j=1}^\infty
j^{-2}\var\left({\bm{\ell}}_{i_j} {{\bm{f}}}_{ t_j}\right) \var
e_{i_j,  t_j}=\sum_{j=1}^\infty j^{-2}\var\left({\bm{\ell}}_{i_j}
{{\bm{f}}}_{ t_j}\right) \psi_1$.  Because the $K$ entries of
${{\bm{f}}}_{ t_j}=[f_{ t_j 1},\ldots,f_{ t_j K}]^\top$ are
independent random variables with unit variance for each $j$, \eqref{L1norm} implies
$\sum_{j=1}^\infty{\var\left({\bm{\ell}}_{i_j}{{\bm{f}}}_{ t_j}\right)}/{j^2}=\sum_{j=1}^\infty{\norm{{\bm{\ell}}_{i_j}}^2}/{j^2}<\infty$. Moreover,
$\Exp({\bm{\ell}}_i{{\bm{f}}}_t
e_{i t})=\Exp({\bm{\ell}}_i{{\bm{f}}}_{t})\Exp( e_{i t})=0$. Hence,
 \eqref{snd} follows from:
\begin{proposition}[\protect{Kolmogorov sufficient condition~\cite[Theorem~6.5.4]{gut2013probability}}]\label{prop:Kolmogorov sufficient condition}
 Let $X_1, X_2, \ldots$ be centered, independent random variables with
 finite variance. Then
$\sum_{j=1}^\infty{
\var X_j}/{j^2} < \infty$ implies $\lim_{n\to\infty}{\sum_{j=1}^n
 X_j}/{n}=0$ (a.s.).
\end{proposition}

The third term of  \eqref{lem:weak} tends almost surely to
$\psi_1$, by the strong law of large numbers. Thus, by \eqref{p0} and \eqref{snd}, we get $\eqref{lem:weak}\mathrel{\stackrel{a.s.}{\to}}
\norm{\bm{\ell}}^2+\psi_1=\sigma_\infty^2$.
 This completes the proof of Lemma~\ref{lem:lim2}.
\end{proof}

\bigskip
Now, we prove Theorem~\ref{thm:equilsd2}. 

\begin{proof} The sample correlation matrix $\mathbf{C}$ and the noncentered sample
 correlation matrix ${\mathbf{\tilde C}}$ (Definition~\ref{def:aux}) are invariant
 under scaling of variables. To the factor model, we can assume
\begin{align}
 \ell_{i t}:=\frac{\ell_{i t}}{\sigma_\infty},\quad
 \psi_1:=\frac{\psi_1}{\sigma_\infty^2}. \label{ass:rescale}
\end{align}
 without loss of
 generality.
Moreover, we can safely assume that $\mu_i=0$ $(1\le i\le N)$ to prove that $F^{\mathbf{C}}$ converges weakly to 
$\mathrm{MP}_{Q,1-\rho}$
 almost surely, because $\mathbf{C}$ is invariant under shifting.
It suffices to show that $\lim_{{\substack{N,T\to\infty\\ {T/N \to Q}}}} \Le\left(F^{\mathbf{C}},\
 F^{\mathbf{E}\mathbf{E}^\top}\right)=0$ (a.s.), because Theorem~\ref{thm:equilsd}~\eqref{assert:EE} implies $\lim_{{\substack{N,T\to\infty\\ {T/N \to Q}}}}\Le\left(F^{\mathbf{E}\mathbf{E}^\top},\, 
\mathrm{MP}_{Q,1-\rho}
\right)=0$ (a.s.).
 By
$\mathbf{C}={\mathbf{Y}}{\mathbf{Y}}^\top$, \cite[Lemma~2.7]{bai99:_method} gives an upper bound on
the fourth power of the L\'evy distance
\begin{align}
\label{levy distance:Cp and EET}
\Le^4\left(F^{\mathbf{C}},\  F^{\mathbf{E}\mathbf{E}^\top}\right)\le \frac{2}{N}\Tr\left({\mathbf{Y}}{\mathbf{Y}}^\top+\mathbf{E}\mathbf{E}^\top\right)\times 
 \frac{1}{N}\Tr\left(\left({\mathbf{Y}} - \mathbf{E}\right)\left({\mathbf{Y}} - \mathbf{E}\right)^\top\right).
\end{align}
On the right side of \eqref{levy distance:Cp and EET},
 $\Tr({\mathbf{Y}}{\mathbf{Y}}^\top)/ N=\Tr\mathbf{C}/ N=1$, and
 \begin{align}
   \frac{\Tr(\mathbf{E}\mathbf{E}^\top)}{N}&=\sum_{i=1}^N 
 {\sum_{t=1}^ T}\frac{\left(x_{i t}-\overline{x}_i\right)^2}{N T}=\sum_{i=1}^N {\sum_{t=1}^ T} \frac{x_{i t}^2}{N T}-\sum_{i=1}^N \frac{\overline{x}_i^2}{N}\le\sum_{i=1}^N {\sum_{t=1}^ T} \frac{x_{i t}^2}{N T}.
\label{levi2c}
 \end{align}
Since the rightmost term of \eqref{levi2c} converges almost surely to a finite deterministic value
by Lemma~\ref{lem:lim2}, $\limsup_{{\substack{N,T\to\infty\\ {T/N \to Q}}}} \abs*{\Tr(\mathbf{E}\mathbf{E}^\top)/N}<\infty$ (a.s.). 
Therefore, for  the right side of \eqref{levy
 distance:Cp and EET}, we have only to assure the following:
\begin{align}
 \label{lev}
 \lim_{{\substack{N,T\to\infty\\ {T/N \to Q}}}} \frac{1}{N}\Tr\left(\left({\mathbf{Y}}-\mathbf{E}\right)\left({\mathbf{Y}} - \mathbf{E}\right)^\top\right)= 0\quad(a.s.).
\end{align}

On the other hand, by ${\mathbf{\tilde C}}=T^{-1}{\mathbf{X}}{\mathbf{X}}$ and \cite[Lemma~2.7]{bai99:_method},
$\Le^4\left(F^{\mathbf{\tilde C}},\  F^{T^{-1}{\mathbf{X}}{\mathbf{X}}^\top}\right)$ is at most the product
 of ${2} N^{-1}\Tr\left({\mathbf{\tilde Y}}{\mathbf{\tilde Y}}^\top+T^{-1}{\mathbf{X}}{\mathbf{X}}^\top\right)$ and
$ N^{-1}\Tr\left(\left({\mathbf{\tilde Y}} -  T^{-1/2}{\mathbf{X}}\right)\left({\mathbf{\tilde Y}} -
  T^{-1/2}{\mathbf{X}}\right)^\top\right)$. Here,   $\Tr {\mathbf{\tilde Y}}{\mathbf{\tilde Y}}^\top/N=\Tr {\mathbf{\tilde C}}/ N=1$
 and $\abs*{\Tr \left((N T)^{-1}{\mathbf{X}}{\mathbf{X}}^\top\right)}<\infty$ (a.s.) by
 Lemma~\ref{lem:lim2}. Hence, to prove that $F^{{\mathbf{\tilde C}}}$ converges weakly to
$\mathrm{MP}_{Q,1-\rho}$ almost surely,
it suffices to guarantee 
\begin{align}
\label{levi}
 \lim_{{\substack{N,T\to\infty\\ {T/N \to Q}}}} \frac{1}{N}\Tr\left(\left({\mathbf{\tilde Y}}- T^{-1/2}{\mathbf{X}}\right)\left({\mathbf{\tilde Y}}- T^{-1/2}{\mathbf{X}} \right)^\top\right)= 0\quad(a.s.),
\end{align}

 We will prove \eqref{levi}, and then derive \eqref{lev} from it.
\medskip

(a) \textit{The proof of~\eqref{levi}}.
The left side of~\eqref{levi} is ${\tilde d}_1-2{\tilde d}_2$ where
\begin{align}\label{levi2a}
 {\tilde d}_1=\frac{1}{N T}\sum_{i=1}^N {\sum_{t=1}^ T} {x_{i t}^2}-1,\quad {\tilde d}_2=
   \frac{1}{N}
 \sum_{i=1}^N \sqrt{\frac{1}{T}{\sum_{t=1}^ T}
 x_{i t}^2} - 1,
\end{align}
because $\Tr({\mathbf{\tilde Y}} {\mathbf{\tilde Y}}^\top)=\Tr({\mathbf{\tilde C}})=N$ by the definition of ${\mathbf{\tilde C}}$ and because
$ \Tr({\mathbf{X}} {\mathbf{\tilde Y}}^\top)=\Tr\left(\left[
 {\mathbf{X}}{\bm{x}_1^\top}/{\norm{\bm{x}_1}},\ldots,{\mathbf{X}}{\bm{x}^\top_N}/{\norm{\bm{x}_N}}\right]\right)=\sum_{i=1}^N 
 {\bm{x}_i\bm{x}_i^\top}/{\norm{\bm{x}_i}}=\sum_{i=1}^N \norm{\bm{x}_i}.$
 Since Lemma~\ref{lem:lim2} implies
 \begin{align}
  \lim_{{\substack{N,T\to\infty\\ {T/N \to Q}}}}{\tilde d}_1=0\quad (a.s.), \label{td1}
\end{align}
   our goal \eqref{levi}
 follows from $\lim_{{\substack{N,T\to\infty\\ {T/N \to Q}}}} {\tilde d}_2=0$ (a.s.). Because of this, we will verify
\begin{align}
 \lim_{{\substack{N,T\to\infty\\ {T/N \to Q}}}} \frac{1}{T}{\sum_{t=1}^ T} x_{i t}^2=1\quad(a.s.).
\label{kpn}
\end{align}

By \eqref{prel}, ${\sum_{t=1}^ T} {x_{i t}^2}/T$ is
 \begin{align}
\frac{\Tr\left(  {\bm{\ell}}_i^\top{\bm{\ell}}_i {\sum_{t=1}^ T}
  {{\bm{f}}}_t{{\bm{f}}}_t^\top \right)}{T}
  +2{\sum_{t=1}^ T}\frac{{\bm{\ell}}_i{{\bm{f}}}_t e_{i t}}{T}
  +{\sum_{t=1}^ T}\frac{ e_{i t}^2-\psi_1}{T}
 + \psi_1.\label{dt}
 \end{align}
 The first term of \eqref{dt} tends almost surely to
 $\norm{\bm{\ell}_i}^2$ in ${{N,T\to\infty,\ {T/N \to Q}}}$,
 by \eqref{uv} and the continuous mapping theorem for the
 continuous function $\Tr$.

Next, we will verify that the second term of \eqref{dt} converges almost
surely to $0$, similarly as \eqref{snd}.  Because $f_{t k}$ $(t\ge1,\
1\le k\le K)$ and $ e_{i t}$ $(t\ge1)$ are independent and centered, so
are ${\bm{\ell}}_i{{\bm{f}}}_{t}$ and $ e_{i t}$.   By $\var f_{t k}=1$ $(1\le k\le K)$, we then have
$\sum_{t=1}^\infty t^{-2}\var\left({\bm{\ell}}_i {{\bm{f}}}_{t} e_{i,
t}\right)=\sum_{t=1}^\infty t^{-2}\var\left({\bm{\ell}}_i
{{\bm{f}}}_{t}\right) \var e_{i, t}=\sum_{t=1}^\infty
 t^{-2}\norm{\bm{\ell}_i}^2  \psi_1<\infty$ by~\eqref{L1norm}. Moreover,
 $\bm{\ell}_i\bm{f}_t e_{i t}$ $(t=1,2,3,\ldots)$ are centered independent.
 Hence, by
Proposition~\ref{prop:Kolmogorov sufficient condition}, the second term
of \eqref{dt} converges almost surely to 0:
\begin{align*}
\lim_{{\substack{N,T\to\infty\\ {T/N \to Q}}}}  2  {\sum_{t=1}^ T}\frac{{\bm{\ell}}_i{{\bm{f}}}_{t} e_{i t}}{T}
  =
2\Exp({\bm{\ell}}_i{{\bm{f}}}_t e_{i t})=2\Exp({\bm{\ell}}_i{{\bm{f}}}_{t})\Exp( e_{i t})=0\quad(a.s.).  
\end{align*}

For
 the third term of \eqref{dt}, we will verify
 \begin{align}
\label{3rd}\lim_{{\substack{N,T\to\infty\\ {T/N \to Q}}}} \max_{{1\le i\le N}}\abs*{T^{-1}{\sum_{t=1}^ T} (e_{i t}^2-\psi_1)}=
 0\quad (a.s.). 
 \end{align}
Because we suppose ${{N,T\to\infty,\ {T/N \to Q}}}\in(0,\,\infty)$, there exists $T_1(m)$ such that for all
 $T$,
\begin{align}
 \label{suppose}T>T_1(m)\implies(Q-1/m)N< T<(Q+1/m)N.
\end{align}
 Since $ e_{i t}^2$
 $(1\le i\le N,\ 1\let\le  T)$ are i.i.d. and $\Exp( e_{11}^2)=\psi_1$,
Proposition~\ref{prop:bai-yin2} implies: for every constant $M>0$, almost surely, there exists $T_2(m,M)$
 such that for all $T\ge T_2(m,M)$ we have $\abs*{\max_{1\le i\le
M T}{\sum_{t=1}^ T} (e_{i t}^2-\psi_1)/ T}<m^{-1}$. Hence, by \eqref{suppose}, for all
 $T\ge\max(T_1(m),\,T_2(m,\, (Q-1/m)^{-1}))$, it holds
 \begin{align*}
  \max_{{1\le i\le \frac{T}{Q+1/m}}}
 \abs*{{\sum_{t=1}^ T}\frac{ e_{i t}^2-\psi_1}{T}}\le\max_{{1\le i\le N}}\abs*{{\sum_{t=1}^ T}\frac{ e_{i t}^2-\psi_1}{T}}\le\max_{{1\le i\le \frac{T}{Q-1/m}}} \abs*{{\sum_{t=1}^ T}\frac{ e_{i t}^2-\psi_1}{T}}<\frac{1}{m}.
 \end{align*}
Consequently,  \eqref{3rd} has verified.

To sum up, by \eqref{ass:rescale} and $\psi=1$. Hence,
 \eqref{kpn} is verified. Thus
 \begin{align}
  \lim_{{\substack{N,T\to\infty\\ {T/N \to Q}}}} {\tilde d}_2=0\quad(a.s.).\label{td2}
 \end{align}
 Therefore \eqref{levi} is concluded.

 \bigskip
 (b) \textit{The proof of~\eqref{lev}}.
First, the left side of~\eqref{lev} is $d_1-2d_2$ where
\begin{align}\label{levi2b}
d_1=\frac{1}{N T}\sum_{i=1}^N {\sum_{t=1}^ T}{\left(x_{i t}-\overline{x}_i\right)^2}-1,\quad d_2=\frac{1}{N}\sum_{i=1}^N \left(\frac{\norm{\bm{x}_i-{\overline{\bm{x}}}_i}}{\sqrt{T}}-1\right).
\end{align}
It is proved in a similar way for (a) as follows: By the definition,
 $\Tr({\mathbf{Y}}{\mathbf{Y}}^\top)=N$.  By Definition~\ref{def:aux}, the $i$-th column
 of ${\mathbf{Y}}^\top$ is ${(\bm{x}_i-{\overline{\bm{x}}}_i)^\top}/\norm{\bm{x}_i-{\overline{\bm{x}}}_i}$, so that of
 $\mathbf{E}{\mathbf{Y}}^\top$ is $ T^{-1/2}({\mathbf{X}}-\overline{\mathbf{X}})\
 {(\bm{x}_i-{\overline{\bm{x}}}_i)^\top}/\norm{\bm{x}_i-{\overline{\bm{x}}}_i}$. Thus, $\Tr\left(\mathbf{E}{\mathbf{Y}}^\top\right)=\sum_{i=1}^N 
{(\bm{x}_i-{\overline{\bm{x}}}_i)(\bm{x}_i-{\overline{\bm{x}}}_i)^\top}/\left(\sqrt{T}\norm{\bm{x}_i-{\overline{\bm{x}}}_i}\right)=\sum_{i=1}^N 
  {\norm{\bm{x}_i-{\overline{\bm{x}}}_i}}/{\sqrt{T}}.$

Secondly, we will show $d_1\mathrel{\stackrel{a.s.}{\to}}0$ and $d_2\mathrel{\stackrel{a.s.}{\to}}0$.
From \eqref{levi2a} and \eqref{levi2b}, it follows immediately
\begin{align}
 \abs*{d_1-{\tilde d}_1}
 &=\frac{1}{N T}\abs*{\sum_{i=1}^N {\sum_{t=1}^ T}
{x_{i t}^2}-\sum_{i=1}^N 
 {\sum_{t=1}^ T}{\left(x_{i t}-\overline{x}_i\right)^2}} =\frac{1}{N}\sum_{i=1}^ N{\abs*{\overline{x}_i}^2}\label{aaa}
\end{align}
Thus,
\begin{align}
 \abs*{d_1-{\tilde d}_1}\le \left(\max_{{1\le i\le N}}{\abs*{\overline{x}_i}}\right)^2\le
 \left(\max_{{1\le i\le N}}\abs*{{\sum_{t=1}^ T}\frac{{\bm{\ell}}_i{{\bm{f}}}_t}{T}}+\max_{{1\le i\le N}}\abs*{{\sum_{t=1}^ T}\frac{ e_{i t}}{T}}\right)^2.\label{target}
 \end{align}
As for the first term inside the right side,
\begin{align*}
 \max_{{1\le i\le N}}\abs*{{\sum_{t=1}^ T}\frac{{\bm{\ell}}_i{{\bm{f}}}_t}{T}}
\le 
\max_{{1\le i\le N}} \left({\sum_{k=1}^K}\abs*{\ell_{ik}}\abs*{{\sum_{t=1}^ T}{f_{t k}}/{T}}\right).
 \end{align*}
By \eqref{L1norm},
\begin{align}
 \max_{{1\le i\le N}}\abs*{{\sum_{t=1}^ T}\frac{{\bm{\ell}}_i{{\bm{f}}}_t}{T}}
\le K L \max_{1\le k\le K}\abs*{{\sum_{t=1}^ T}\frac{f_{t k}}{T}}\mathrel{\stackrel{a.s.}{\to}}0\quad\left({{N,T\to\infty,\ {T/N \to Q}}}\right), \label{ppp}
\end{align} 
by the strong law of large numbers, because for each $k$ $(1\le k\le K)$, $f_{t k}=0$ $(t\ge1)$ are centered i.i.d.

Next, we consider the second term $\max_{{1\le i\le N}}\abs*{{\sum_{t=1}^ T} e_{i t}/{T}}$ inside the right side of \eqref{target}.
By Proposition~\ref{prop:bai-yin2}, for every constant $M>0$, almost surely, there exists $T_3(m,M)$
 such that for all $T\ge T_3(m,M)$ we have $\abs*{\max_{1\le i\le
M T}{\sum_{t=1}^ T} e_{i t}/ T}<m^{-1}$. Hence, by \eqref{suppose}, for all
 $T\ge\max(T_1(m),\,T_3(m,\, (Q-1/m)^{-1}))$, it holds
 \begin{align*}
  \max_{{1\le i\le \frac{T}{Q+1/m}}}
 \abs*{{\sum_{t=1}^ T}\frac{ e_{i t}}{T}}\le\max_{{1\le i\le N}}\abs*{{\sum_{t=1}^ T}\frac{ e_{i t}}{T}}\le\max_{{1\le i\le \frac{T}{Q-1/m}}} \abs*{{\sum_{t=1}^ T}\frac{ e_{i t}}{T}}<\frac{1}{m}.
 \end{align*}
Consequently, $\displaystyle\lim_{{\substack{N,T\to\infty\\ {T/N \to Q}}}} \max_{{1\le i\le N}}\abs*{T^{-1}{\sum_{t=1}^ T} e_{i t}}=
 0$ (a.s.). 
By this and \eqref{ppp},
\begin{align}
\abs*{d_1-{\tilde d}_1}\le\left(\max_{{1\le i\le N}}\abs*{{\sum_{t=1}^ T}\frac{{\bm{\ell}}_i{{\bm{f}}}_t}{T}}+\max_{{1\le i\le N}}\abs*{{\sum_{t=1}^ T}\frac{ e_{i t}
 }{T}}\right)^2\mathrel{\stackrel{a.s.}{\to}} 0\quad({{N,T\to\infty,\ {T/N \to Q}}}).
\label{levi2}
\end{align}
Thus, $d_1\mathrel{\stackrel{a.s.}{\to}} 0$, by \eqref{td1}. 

In contrast, by \eqref{kpn} and \eqref{levi2b},
\begin{align*}
\abs*{d_2-{\tilde d}_2}&\le\frac{1}{N\sqrt{T}}\sum_{i=1}^N \abs*{\sqrt{{\sum_{t=1}^ T}(x_{i t}-\overline{x}_i)^2}-\sqrt{{\sum_{t=1}^ T} x_{i t}^2}}.\end{align*}
Note $|\sqrt{r}-\sqrt{s}|\le |r-s|/\sqrt{s}$ $(r,s\ge0)$. Let
 $r={T}^{-1}{\sum_{t=1}^ T} (x_{i t}-\overline{x}_i)^2$ and $s=
 {T}^{-1}{\sum_{t=1}^ T} x_{i t}^2$. Then, 
$|r-s|=(\overline{x}_i)^2$. 
\begin{align*}
\abs*{d_2-{\tilde d}_2}\le N^{-1}\sum_{i=1}^N \abs*{\overline{x}_i}^2 \cdot \left({\sum_{t=1}^ T}{x_{i t}^2}/{T}\right)^{-1/2}.
\end{align*}
On the right side, $ N^{-1}\sum_{i=1}^N \abs*{\overline{x}_i}^2$ 
is $\abs*{d_1-{\tilde d}_1}$ by \eqref{aaa}, which converges almost surely
to 0 by \eqref{levi2}.  
Thus,  \eqref{kpn} implies $\abs*{d_2-{\tilde d}_2}\mathrel{\stackrel{a.s.}{\to}}0$, from which $d_2\mathrel{\stackrel{a.s.}{\to}}
 0$ follows  by \eqref{td2}. Therefore, we have demonstrated
 \eqref{lev}. This completes the proof of Theorem~\ref{thm:equilsd2}.
\end{proof}

\section*{Acknowledgments}
Kazuyoshi Yata and Makoto Aoshima kindly drew the author's attention to
\cite{yata09:_pca_consis_non_gauss_data,YATA20102060,YATA2012193} and
\cite{ishii21:_hypot}. The author thanks the reviewers.

\end{document}